\documentclass[a4paper,12pt]{article}
\usepackage[english]{babel}
\usepackage[T1]{fontenc}
\usepackage[utf8]{inputenc}
\usepackage[obeyspaces]{url}

\usepackage{mathtools, amsmath, amsthm, amssymb, mathtools, 
tikz,tikz-cd, subcaption, framed, float, tensor, bm, dsfont, 
floatpag, microtype, fullpage, mathbbol}

\usepackage[normalem]{ulem}

\usepackage[
  bookmarks=false, 
  colorlinks,
  citecolor=blue,
  linkcolor=blue,
  urlcolor=blue,
]{hyperref}
\usepackage[affil-it]{authblk}
\usepackage{cleveref}
\usetikzlibrary{shapes.geometric, backgrounds, calc,
  arrows, automata, fit, positioning, patterns,
  decorations.pathreplacing, tikzmark, arrows.meta}

\numberwithin{equation}{section}
\numberwithin{figure}{section}

\newtheorem{thm}{Theorem}[section]
\newtheorem{lem}[thm]{Lemma}
\newtheorem{cor}[thm]{Corollary}

{
\theoremstyle{definition}
\newtheorem{defn}[thm]{Definition}
}

{
\newtheorem{rem}[thm]{Remark}
}

\newcommand{\bb}[1]{\bm{\mathrm{#1}}}
\newcommand{\m}{\bb}
\newcommand{\mc}[1]{\m{\hat{#1}}}
\newcommand\oeis[1]{\href{https://oeis.org/#1}{#1}}

\DeclareMathOperator{\pk}{\m{pk}}
\DeclareMathOperator{\dd}{\m{dd}}
\DeclareMathOperator{\eOdis}{\m{eOdis}}
\DeclareMathOperator{\oOdis}{\m{oOdis}}
\DeclareMathOperator{\eEdis}{\m{eEdis}}
\DeclareMathOperator{\oEdis}{\m{oEdis}}
\DeclareMathOperator{\aOrpt}{\m{aOrpt}}
\DeclareMathOperator{\aErpt}{\m{aErpt}}

\newcommand{\F}{\mathcal{F}}

\newcommand{\N}{\mathbb{N}}
\newcommand{\Z}{\mathbb{Z}}

\newcommand{\A}{\mathcal{A}}
\newcommand{\B}{\mathcal{B}}
\newcommand{\s}{\mathcal{S}} 

\newcommand{\D}{\mathcal{D}}

\newcommand{\p}{\mathcal{P}} 
\newcommand{\I}{\mathcal{I}}
\newcommand{\LAT}{\mathcal{L}}

\def\udots{\reflectbox{$\ddots$}}

\let\le\leqslant
\let\ge\geqslant
\let\leq\leqslant
\let\geq\geqslant

\DeclareUnicodeCharacter{21A6}{$\mapsto$}

\author[1]{Jean-Luc Baril}
\author[2]{Alexander Burstein}
\author[1]{Sergey Kirgizov}
\affil[1]{\rm LIB, Université de Bourgogne Franche-Comté\protect\\
 B.P. 47 870, 21078 Dijon Cedex France\protect\\
  {\tt E-mails: \{barjl, sergey.kirgizov\}@u-bourgogne.fr
  }
\medskip
}
\affil[2]{\rm Howard University, Department of Mathematics\protect\\ Washington, DC 20059, USA\protect\\
  {\tt E-mail: aburstein@howard.edu}}
\date{\today}

\title{Pattern statistics in faro words and permutations}

\begin{document}

\maketitle

\begin{abstract}
 We study the distribution and the popularity of some patterns in {\em $k$-ary faro words}, 
i.e. words over the 
 alphabet $\{1, 2, \ldots, k\}$ obtained by interlacing 
 the letters of two nondecreasing  words of lengths differing by at most one.
 We present a bijection between these words and dispersed 
 Dyck paths (i.e. Motzkin paths with all level steps 
  on the $x$-axis)
 with a given number of peaks. We show
 how the bijection maps statistics of
 consecutive patterns of faro words 
 into linear combinations of
 other pattern statistics on paths.  Then, we deduce
 enumerative results by providing multivariate 
 generating functions for the distribution and 
 the popularity of patterns of length at most 
 three. Finally, we consider some interesting 
 subclasses of faro words that are permutations, 
 involutions, derangements, or subexcedent words.
\end{abstract}


\section{Introduction and notations}

The faro shuffle is a well-known technique to shuffle a deck of
cards. The deck is split in two at the middle, and the
cards from the two halves are combined back by taking alternatively
the bottoms of stacks. Certain mathematical questions about the faro
shuffle are considered for example in the works of Morris~\cite{Mor}, 
Diaconis, Graham and Kantor~\cite{Dia}.
Inspired by these studies and a solid body of modern combinatorial
literature (see for instance Lothaire~\cite{lot}, Stanley~\cite{sta},
Bóna~\cite{bona} and Kitaev~\cite{Kit} books) that explores
enumerative and bijective aspects of patterns in various discrete
structures, the present paper considers an unexpectedly overlooked
combinatorial objects, which we call {\em faro words}. They are
special kind of word shuffles, which are important in several algorithmic and combinatorial settings (see for example Barnes work~\cite{barnes} and references therein).
In this paper, we present enumerative results and show how
faro words and patterns therein are related to other structures such
as Dyck paths, Motzkin paths and Dumont permutations.

\subsection{Faro words and permutations}
We deal with $k$-ary words $u_1 u_2 \ldots u_n$ over the
integer alphabet $[1,k] = \{1, 2, \ldots, k\}$ endowed with the usual
total order. A $k$-ary word is called {\em nondecreasing} if $u_i \le
u_{i+1}$ for all $i \in [1, n-1]$. 
\begin{defn}\label{faroword} For two $k$-ary words $u$ and $v$
such that $0\leq |u|-|v|\leq 1$, the {\it faro shuffle} of $u$ and $v$
is the $k$-ary word of length $|u|+|v|$ obtained by interlacing the
letters of $u$ and $v$ as follows: $u_1v_1u_2v_2u_3v_3\ldots$ A
\emph{$k$-ary faro word} is a faro shuffle of two nondecreasing
$k$-ary words. 
\end{defn}
Let $\s_{n,k}$ be the set of $k$-ary faro words of
length $n$. Its cardinality equals the product of two binomial
coefficients $\binom{\lfloor n/2 \rfloor + k-1}{k-1}\binom{\lceil n/2
  \rceil + k-1}{k-1}$, each of them being, respectively, the number of
$m$-multisets of $[1,k]$ for $m=\lfloor\frac{n}{2}\rfloor$ and
$m=\lceil\frac{n}{2}\rceil$.  For example, we have $\s_{4,2} = \{1111,
1112, 1121, \allowbreak 1122, \allowbreak 1212,\allowbreak 1222, 2121,
2122, 2222\}$ and $|\s_{4,2}|=9$. 
\begin{defn}\label{faroperm} A \emph{faro permutation} of length
 $n$ is an $n$-ary faro word of length $n$ that contains
every letter in $[1,n]$ exactly once. 
\end{defn}
Let $\p_n$ be the
set of length $n$ faro permutations. For instance, we have
$\p_3=\{123,132,213\}$. Since a faro permutation is entirely
determined by the choice of its values on the odd indices, the
cardinality of $\p_n$ is $\binom{n}{\left\lfloor n/2\right\rfloor}$.

A $k$-ary word $w = w_1 w_2 \ldots w_n$ avoids a classical pattern
(resp. consecutive pattern) $p = p_1\mbox{-} p_2\mbox{-} \cdots\mbox{-} p_k$
(resp.  $p = p_1p_2 \ldots p_k$) if there does not
 exist a strictly increasing sequence of indices $i_1
i_2 \ldots i_k$ (resp. with $i_{j+1}=i_{j}+1$ for $1\leq j\leq k-1$)
such that $w_{i_1} w_{i_2} \ldots w_{i_k}$ is order-isomorphic to $p$
(see \cite{Kit} for instance). Obviously, any faro word avoids the
classical pattern $3\mbox{-}2\mbox{-}1$. Let $Av_n(\sigma)$
denote the set of permutations avoiding a classical pattern $\sigma$,
then we have $\p_n\subseteq Av_n(3\mbox{-}2\mbox{-}1)$ for $n\geq 0$, and $\p_n\neq
Av_n(3\mbox{-}2\mbox{-}1)$ for $n\geq 3$ since $(n-1)n12\ldots (n-2)\in Av_n(3\mbox{-}2\mbox{-}1)$ is not a faro word.  Note that a faro permutation can contain all classical patterns of length $3$ except $3\mbox{-}2\mbox{-}1$ (e.g., $31425$). 
\begin{rem}\label{excpat} A $k$-ary word $w = w_1 w_2 \ldots w_n$ is a faro word if and only if $w_i \le w_{i+2}$ for any $i \in [1,n-2]$, which means that faro permutations are precisely those avoiding the three consecutive patterns $231$, $321$ and $312$.
\end{rem}

\subsection{Dyck and dispersed Dyck paths}
In order to study the distribution of patterns in faro words, we will exhibit one-to-one correspondences between these objects and some specific lattice paths in the first quadrant of the plane. Hence, we provide basic necessary definitions on lattice paths.
\begin{defn}\label{paths}
\emph{Dispersed Dyck paths} (see \cite{hac}) are lattice paths starting at $(0,0)$, ending at $(n,0)$, consisting of level steps $F=(1,0)$, up step $U = (1,1)$ and down steps $D = (1,-1)$, and never going below the $x$-axis and where all level steps are on the $x$-axis. 
\end{defn}
 Let $\B_n$ be
 the set of dispersed Dyck paths of length $n$ (or, equivalently, consisting of $n$ steps) and set $\B=\cup_{n\geq 0}\B_n$, where the empty path is denoted by $\epsilon$. A \emph{Dyck path} of semilength $n \ge 0$ is a dispersed Dyck path of length $2n$ with no level steps. Let $\D_n$ be the set of Dyck paths of semilength $n$ and let $\D = \bigcup_{n\ge0}\D_n$. Dispersed Dyck paths of length $n$ are in straightforward bijection with prefixes of Dyck paths of length $n$, also known as ballot paths~\cite{Ber,Whi}. Indeed, we can obtain a ballot path from a dispersed Dyck path by
 replacing all level steps with up steps.  Dyck and dispersed Dyck paths are counted by the Catalan and ballot numbers, respectively (see
  \oeis{A000108} and \oeis{A001405} in the Online
 Encyclopedia of Integer Sequences of N.J.A. Sloane~\cite{Slo}, where the general terms are $c_n = \frac{1}{n+1}\binom{2n}{n}$ and $b_n = \binom{n}{\lfloor n/2 \rfloor}$, respectively).

A path $P$ avoids a pattern $X$ if and only if $P$ does not contain $X$ as a sequence of consecutive steps (see for instance \cite{deu,mansour}). Note that other pattern definitions exist in the literature where steps are not necessarily consecutive~\cite{Bach}. We also need some notations similar to Kleene star and plus symbols
of formal language theory. For a nonempty pattern $X$, an occurrence of the pattern $X^+$ in a path $P$ is a maximal sequence of consecutive  repetitions of $X$, i.e. a maximal subword of the form $X^k$ for $k\geq 1$. The pattern $X^*$ will be either an empty pattern or a pattern $X^+$. More generally, for two possibly empty patterns $Y$ and $Z$ such that $Y$ does not end with $X$ and $Z$ does not start with $X$, the pattern $YX^+Z$ (resp. $YX^*Z$) corresponds to an occurrence  obtained by concatenation of $Y$, $X^+$ and $Z$ (resp. $Y$, $X^*$ and $Z$). For instance, the path $FUDUDFFUDF$ contains two occurrences of the pattern $F(UD)^+F$ and three occurrences of $F(UD)^*F$.

\subsection{Statistics on words and lattice paths}
\begin{defn} \label{statistic} 
A \emph{statistic} $\m{s}$ is an integer-valued function from a set
$\A$ of words or paths.
\end{defn}
To a given pattern $p$, we associate the
pattern statistic $\m{p} : \A \to \N$ such that $\m p(a)$ is the
number of occurrences of the pattern $p$ in the object $a \in \A$ (we
use the boldface to denote statistics). For example,
  the statistic giving the number of
  occurrences of the consecutive pattern $123$ (resp. $UDUD$) in a word (resp. a
  lattice path) is denoted by $\m{123}$ (resp. $\m {UDUD}$).
  We denote by $\mc 1$ (resp. $\mc 2, \mc n$) the constant statistic
  returning the value $1$ (resp. $2$, $n$).
 \begin{defn} The \emph{popularity} of a pattern $p$ in $\A$ is the total number of
  occurrences of $p$ over all objects of $\A$, that is
  $\m{p}(\A)=\sum_{a\in \A}\m p(a)$ (see ~\cite{des, Bon, Homb, Kit}).
  \end{defn} 
  For instance, for a dispersed Dyck path $P = FFUDFUUDUUUDDDD$ we have
  $\m{FF}(P) = 1$, $\m{DDD}(P) = 2$, $\m{UD}(P) = 3$, $\m{UUUU}(P)
  = 0$ and $\mc 1 (P) = 1$.
  Moreover, if $\A=\{UUDD,UDUD\}$ then the popularity of the
  pattern $UD$ in $\A$ is $\m{UD}(\A)=3$.

Let $\m T_\A$ be the set of all statistics defined on a set $\A$.  For any pair of statistics $\m s,\m t \in \m T_\A$, we define the statistic $\m s + \m t$ by $(\m s+\m t)(a)=\m s(a)+\m t(a)$ for any $a\in \A$, which endows $\m T_\A$ with a $\Z$-module structure. Let $\B$ be a set of combinatorial objects, and let  $\m T_\B$ be the corresponding set of statistics. We say that two statistics $\m s \in \m T_\A$ and $\m t \in \m T_\B$ have the \emph{same distribution}, or are \emph{equidistributed}, if there exists a bijection $f : \A \to \B$ such that $\m s(a) = \m t (f (a))$ for any $a \in \mathcal{A}$. In this case, with a slight abuse of the notation already used in \cite{patdd}, we write shortly  $f(\m s)=\m t$ or $\m s =\m t$ whenever $f$ is the identity. As a byproduct, for any constant statistic $\mc n$, we have  $f(\mc n) = \mc n$.

\subsection{Outline of the paper}
The paper is organized as follows. In Section \ref{sec:pat_faro_words}, we present a constructive bijection $f$ between the set $\s_{n,k}$ of $k$-ary faro words of length $n$ and the set of dispersed Dyck paths of length $n+2k-2$ with $k-1$ peaks. We show where pattern statistics are transported by $f$, which provides a more suitable ground for studying the distribution of consecutive patterns. Thus, we derive enumerating results on the distribution and popularity of patterns in $\s_{n,k}$ by giving multivariate generating functions
where the coefficient of $x^ny^kz^t$ is the number of $k$-ary faro words of length $n$ having exactly $t$ occurrences of a given pattern.  In Section \ref{sec:pat_faro_perm}, we present a similar study for faro permutations.  More precisely, we provide a bijection $g$ between $\p_n$ and the set of dispersed Dyck paths of length $n$ and show how $g$ acts on pattern statistics of length at most three. Consequently, we deduce enumerative results for the distribution and the popularity of these patterns in $\p_n$.  We also present a bijection between $\p_n$ and involutions avoiding the classical pattern $3\mbox{-}2\mbox{-}1$.  Finally, in  Section 4, we prove that the set of subexcedent words in $\s_{n,n}$ is related to ternary trees and Dumont permutations of the second kind~\cite{Bur} avoiding
the classical pattern $2\mbox{-}1\mbox{-}4\mbox{-}3$, and we show why faro involutions and faro derangements are respectively enumerated by the Fibonacci and Catalan numbers. 

\section{Patterns in faro words} \label{sec:pat_faro_words}

In this section we construct a bijection $f$ between the set $\s_{n,k}$ of $k$-ary faro words of length $n$ and a subset of dispersed Dyck paths, and show how $f$ transports pattern statistics.  Then, we deduce generating functions for the distribution and popularity of some patterns.

A \emph{pair} in a faro word $w$ is an occurrence $w_iw_{i+1}$ with $w_i > w_{i+1}$.  Remark~\ref{excpat} implies that a letter cannot be part of two pairs since a faro word avoids the consecutive pattern $321$. A \emph{singleton} in $w$ is a letter $w_i$ not in any pair of $w$. Any faro word can be uniquely decomposed as a sequence of pairs and singletons, which are called \emph{blocks} of faro words. For instance, the block decomposition of $111212131333$ is $1^3 (21)^2 (31) 3^3$.

Let $\LAT_k$ be the set of all possible blocks of a decomposition of a $k$-ary faro word, that is
$$
\LAT_k=\{1,2,\ldots ,k\} \cup \{ji \,:\, 1\leq i<j\leq k\}.
$$

\begin{defn}\label{order} We define an order relation $\preceq$ on $\LAT_k$ as follows: for
$g,h,i,j\in\{1,2,\ldots,k\}$, 
$$\left\{\begin{array}{ll}
i \preceq j, &\mbox{ if } i \leq j,\\
i \preceq jh, &\mbox{ if } i \leq h < j,\\
ig \preceq j, &\mbox{ if } g < i \leq j,\\
 ig \preceq j h, &\mbox{ if } g<i\leq j  \mbox{ and } g\leq h<j.\\
\end{array}\right.
$$ 
\end{defn}

\begin{rem}\label{remorder} The order relation $\preceq$ can be defined less technically as follows:
for
$p,q\in\LAT_k$, 
$$
p\preceq q \Longleftrightarrow pq\mbox{ is a faro word different from a pair.}
$$ 
\end{rem}
This order relation endows the set $\LAT_k$ with a poset structure, which we call \emph{faro poset}. See Figure~\ref{jl} for an illustration of the Hasse diagram of $(\LAT_k,\preceq)$.  

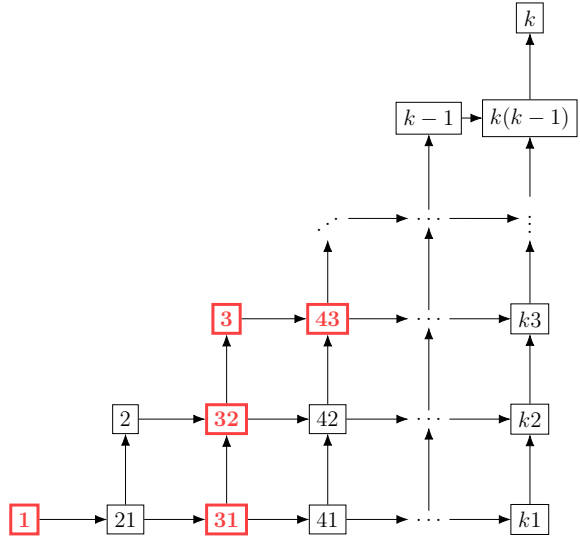
\begin{figure}[!ht]
  \centering
  \scalebox{0.7}{\begin{tikzpicture}[scale=1.9]
  \definecolor{Red2}{HTML}{FA4040}
  \tikzset{
    b/.style={circle, fill, minimum size=1mm, inner sep=0pt, draw=white},
    sp/.style={line width=1mm, white},
    s/.style={rectangle, draw},
    r/.style={rectangle, draw, Red2, ultra thick, font=\boldmath},
    every label/.style={label distance=5mm}
  }

  \node[r] (a1)  at (1,1) {$1$};
  \node[s] (a21) at (2,1) {$21$};
  \node[r] (a31) at (3,1) {$31$};
  \node[s] (a41) at (4,1) {$41$};
  \node    (a51) at (5,1) {$\ldots$};
  \node[s] (a61) at (6,1) {$k1$};

  \node[s] (a2)  at (2,2) {$2$};
  \node[r] (a32) at (3,2) {$32$};
  \node[s] (a42) at (4,2) {$42$};
  \node    (a52) at (5,2) {$\ldots$};
  \node[s] (a62) at (6,2) {$k2$};

  \node[r] (a3)  at (3,3) {$3$};
  \node[r] (a43) at (4,3) {$43$};
  \node    (a53) at (5,3) {$\ldots$};
  \node[s] (a63) at (6,3) {$k3$};

  \node (a4)  at (4,4) {$\udots$};
  \node (a54) at (5,4) {$\ldots$};
  \node (a64) at (6,4) {$\vdots$};

  \node[s] (a5)  at (5,5) {$k-1$};
  \node[s] (a65) at (6,5) {$k(k-1)$};

  \node[s] (a6)  at (6,6) {$k$};

  \path[-{Latex[scale=1.3]}] (a1) edge (a21)
  (a21) edge (a31) edge (a2)
  (a31) edge (a41) edge (a32)
  (a41) edge (a51) edge (a42)
  (a51) edge (a61) edge (a52)
  (a61) edge (a62)

  (a2) edge (a32)
  (a32) edge (a3) edge (a42)
  (a42) edge (a43) edge (a52)
  (a52) edge (a62) edge (a53)

  (a62) edge (a63)

  (a3) edge (a43)
  (a43) edge (a53) edge (a4)
  (a53) edge (a63) edge (a54)
  (a63) edge (a64)

  (a4) edge (a54)
  (a54) edge (a64) edge (a5)
  (a64) edge (a65)

  (a5) edge (a65)

  (a65) edge (a6)

  ;
\end{tikzpicture}}
  \caption{The faro poset $(\LAT_k,\preceq)$. Red blocks represent the multichain associated to the  $k$-ary
    faro word $11313232343=1^2 (31) (32)^2 3 (43)$.}
  \label{jl}
\end{figure}

A \emph{multichain} in a poset is a chain, i.e. \emph{a totally ordered subset}, with repetitions allowed. Due to the simple structure of the faro poset, we easily deduce the following remarks.

\begin{rem} \label{uni} 
There is a one-to-one correspondence between $k$-ary faro words and
the multichains of $\LAT_k$. Indeed, Remark~\ref{remorder} implies that the block decomposition of a
$k$-ary faro word $w$ into pairs and singletons $w = b_1 b_2 \ldots b_\ell$
unambiguously corresponds to the multichain $b_1 \preceq b_2
\preceq \cdots \preceq b_\ell$ in $\LAT_k$, and vice versa.
For instance, the faro word $11313232343=11(31)(32)(32)3(43)$ corresponds to the multichain $1\preceq 1\preceq 31\preceq 32\preceq 32\preceq 3\preceq 43$ (see Figure \ref{jl}).
\end{rem}

\begin{rem} \label{unibis}
If a $k$-ary faro word $w$ contains a singleton $x$ in its decomposition into blocks, then it satisfies the following property: the set of pairs of the form $ab$, $b<a\leq x$, equals the set of pairs of the form $cd$, $d\leq x-1$.
\end{rem}

\subsection{A bijection to the set of dispersed Dyck paths}

As mentioned by E. Deutsch in \cite{Slo} (see sequence
\oeis{A124428}), the number of dispersed paths of length $n$ with $k$ peaks (a peak is an occurrence of the pattern $UD$) is given by 
$$
|\B_{n,k}| = 
\binom{\left\lfloor\frac{n}{2}\right\rfloor}{k}
\binom{\left\lceil\frac{n}{2}\right\rceil}{k}.
$$
Thus, we present a bijection $f$ from the set $\s_{n,k}$ of $k$-ary faro words of length $n$ to the set $B_{n+2(k-1), k-1}$ of dispersed Dyck paths of length $n+2(k-1)$ with exactly $k-1$ peaks. For a given $w\in \s_{n,k}$, we set 
$$
f(w)=F^{T_0} U^{T_1} D^{T_2} F^{T_3} \ldots F^{T_{3(k-2)}}
U^{T_{3(k-2)+1}} D^{T_{3(k-2)+2}} F^{T_{3(k-1)}},
$$
where $T_i$ is defined for $0\leq i\leq 3(k-1)$ as follows:
\begin{itemize}

\item if $i=3(x-1)$ then $T_i$ is the number of occurrences of the singleton $x$ in $w$;

\item if $i=3(x-1)-1$ then $T_i$ is one plus the number of pairs $xy$, $y<x$, in $w$;

\item if $i=3(x-1)+1$ then $T_i$ is one plus the number of pairs $yx$, $y>x$, in $w$.

\end{itemize}

It is worth noting that the image of a faro word $w\in \s_{n,k}$
depends on the arity $k$ that we consider. Indeed, the image of the
empty word $\epsilon$ is $UD$ when $k=2$, while $f(\epsilon)=UDUD$ for
$k=3$. We refer to Figure~\ref{bij} for one detailed example of this
bijection, while Figure~\ref{patex} provides more additional examples.
For instance, the images by $f$ of the $5$-ary
words $\epsilon, 12345, 3141, 111111212222$ are, respectively,
$UDUDUDUD$, $FUDFUDFUDFUDF$, $UUUDUDDUDDUD$ and
$FFFFFFUUDDFFFFUDUDUD$.

\begin{figure}[!ht]
  \centering
  \scalebox{1.1}{
  \def\svgwidth{30em}
  \input{bijpaper3.tex}
  }
  \caption{The image by  $f$ of the $5$-ary faro word $w=11313232343$ is $f(w)=FFUUDUUUDDDDFUUDDUD$.}
  \label{bij}
\end{figure}

\begin{rem} \label{linearalgo}
Clearly, the values $T_i$, $0\leq i\leq 3(k-1)$, can be obtained 
from $w$ by reading it from left to right and by determining if the current
entry $x$ belongs to either a pair $xy$ or $yx$, or a singleton $x$. Moreover, values of $T$ at indices $i = 0 \mod 3$ correspond 
to the lengths of maximal runs of consecutive level 
steps, and values at indices $i = 1 \mod 3$ (resp. $i = 
2 \mod 3$) correspond to the lengths of maximal runs
of consecutive up (resp. down) steps, which means that the sequence $T=T_0T_1\ldots T_{3(k-1)}$ is a 
run-length-like encoding of the path $f(w)$. Thus, $f(w)$ can be constructed from $w$ using a linear time algorithm. 
\end{rem}
\begin{lem}\label{bijimage} The path $f(w)$ is necessarily a dispersed Dyck path of
length $n+2(k-1)$ with exactly $k-1$ peaks.
\end{lem}
\begin{proof}
Since for any 
$i\neq 0 \mod 3$, $1\leq i\leq
3(k-1)-1$ we have $T_i\geq 1$, the path $w$ contains
exactly $k-1$ peaks $UD$.
Interpreting Remark
\ref{unibis} on the path $f(w)$, the number of up
steps before a given level step equals the number of
down steps before the same level step, which implies
that any level step belongs to the $x$-axis.  Let
$d_x=\sum_{i=2}^{x+2}T_{3(i-1)-1}$
(resp. $u_x=\sum_{i=1}^{x+1}T_{3(x-1)+1}$) be the
total number of down steps (resp. up steps) in the
first $x+1$ maximal runs of down steps (resp. up
steps). Due to the definition of $f$, $d_x$ equals
the number of pairs $ij$, $1\leq j<i\leq x+2$, in
$w$, and $u_x$ equals the number of pairs $ij$,
$1\leq j\leq x+1$, $i\geq j+1$, which implies that
$d_x\leq u_x$.  Also by definition, the total number
of up steps (resp. down steps) in $f(w)$ equals the
total
number of pairs in $w$, which completes the proof.
\end{proof}

\begin{thm} \label{farobij}
The map $f$ is a bijection from $\s_{n,k}$ to the set $\B_{n+2(k-1),k-1}$ of dispersed Dyck paths of length $n + 2(k-1)$ with exactly $k-1$ peaks.
\end{thm}

\begin{proof}
Let us prove that if $w$ and $w'$ are two distinct $k$-ary faro words then we have $f(w)\neq f(w')$.  Let $i\geq 1$ be the smallest positive integer such that $w_i\neq w'_i$. Without loss of generality, we assume $w_i<w'_i$. Let us consider the positions of $w_i$ and $w'_i$ in the block 
decomposition of $w$.

If $w_i$ and $w'_i$ are both in the pairs $w_{i}w_{i+1}$ and $w'_{i}w'_{i+1}$, then Remark~\ref{uni} implies that a pair $w_ix$, $w_i>x$, cannot appear to the right of $w'_{i}$ in $w'$, which implies that $T_{3(w_i-1)-1}\neq T'_{3(w_i-1)-1}$, and thus $f(w)\neq f(w')$.
    
There remain the following cases: 
\begin{enumerate}
\item[(i)] $w_i$ or $w'_i$ is a singleton in $w$, 

\item[(ii)] $w_i$ and $w'_i$ are both in the pairs
$w_{i-1}w_i$ and $w'_{i-1}w'_i=w_{i-1}w'_i$,

\item[(iii)] $w_i$ belongs to the  pair $w_{i-1}w_i$ and $w'_i$ belongs to the pair $w'_iw'_{i+1}$,

\item[(iv)] $w_i$ and $w'_i$ are both in the pairs $w_{i}w_{i+1}$ and $w'_{i-1}w'_{i}$. 
\end{enumerate}
The fact that a faro word avoids $231$ in case (i) and Remark~\ref{uni} for cases (ii), (iii) (iv), imply that $w_i$ cannot appear to the right of $w'_i$ in $w'$.  Then the number of $w_i$ in $w$, i.e. $T_{3(w_i-1)}+T_{3(w_i-1)+1}+T_{3(w_i-1)-1}$, is different from the number of $w_i$ in $w'$, which is $T'_{3(w_i-1)}+T'_{3(w_i-1)+1}+T'_{3(w_i-1)-1}$. Therefore, there is $\delta\in\{-1,0,1\}$ such that $T_{3(w_i-1)+\delta}\neq T'_{3(w_i-1)+\delta}$, which implies that $f(w)\neq f(w')$.

Thus, $f$ is an injective map, and using a cardinality argument (see \oeis{A124428} in \cite{Slo}), we conclude that $f$ is a bijection from $\s_{n,k}$ to $\B_{n+2(k-1),k-1}$.
\end{proof}

 Although it is not used in the paper, we could prove that from a given dispersed Dyck path 
$P\in\B_{n+2(k-1),k-1}$, $f^{-1}(P)$ can be obtained after applying the following procedure. We refer to Figure~\ref{patex} for several examples.

We set $s=1$ as the initial value. We mark all 
$D$-steps preceded by an $U$-step and all the other 
$D$-steps are left unmarked. Reading the steps of 
$P$ from left to right:
\begin{itemize}
\item If a $D$-step is encountered, then skip it.
\item If an $F$-step is encountered, then write the singleton $s$. If the next step is not an $F$-step, then update $s = s + 1$.
\item If an $U$-step is encountered in the $i$th run of $U$-steps, then we distinguish two cases:
\begin{enumerate}
\item[(i)] the next step is $D$; then we skip this $UD$-pattern by continuing from the step after $D$, if it exists.

\item[(ii)] the next step is $U$; then we write the pair $ji$, where $j$ is the least integer such that the $(j-1)$-th run of $D$-steps has at least one unmarked $D$-step. Mark the first unmarked $D$-step from the $(j-1)$-th run of $D$-steps.
\end{enumerate}
\end{itemize}

\begin{figure}[!ht]
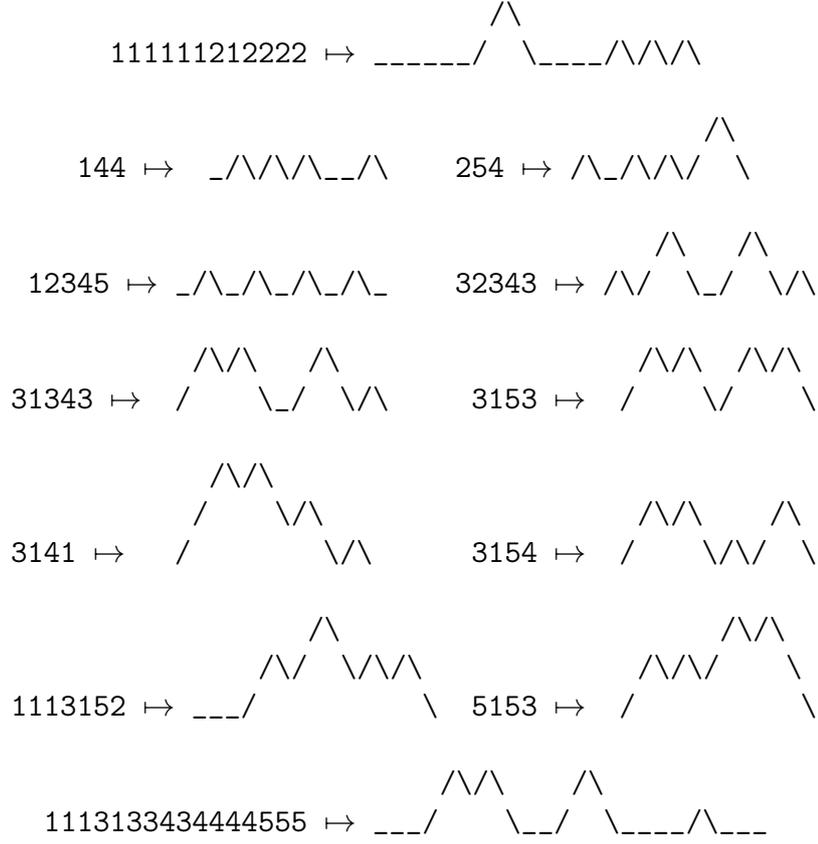

  \centering
\begin{verbatim}
                                       /\           
                111111212222 ↦ ______/  \____/\/\/\

                                                    /\
              144 ↦  _/\/\/\__/\    254 ↦ /\_/\/\/  \

                                                 /\   /\  
           12345 ↦ _/\_/\_/\_/\_    32343 ↦ /\/  \_/  \/\

                     /\/\   /\                  /\/\  /\/\   
          31343 ↦  /    \_/  \/\     3153 ↦  /    \/    \

                      /\/\      
                     /    \/\                   /\/\    /\   
          3141 ↦   /        \/\      3154 ↦  /    \/\/  \

                            /\                       /\/\  
                         /\/  \/\/\             /\/\/    \ 
          1113152 ↦ ___/          \  5153 ↦  /          \

                                    /\/\    /\          
            1113133434444555 ↦ ___/    \__/  \____/\___

\end{verbatim}
 \caption{Images of several 5-ary words under bijection $f$.}
 \label{patex}
\end{figure}

\subsection{Distribution and popularity of patterns}

In this part, we first show how the bijection $f$ transports pattern statistics on $\s_{n,k}$ into the context of dispersed Dyck paths. After, we deduce multivariate generating functions for the distribution and the popularity of patterns of length two by exploiting the classic recursive decomposition of dispersed Dyck paths.

\begin{thm}
  \label{faropat}
  For $n\geq 0$, the bijection $f$ from $\s_{n,k}$ to
  $\B_{n+2(k-1),k-1}$ maps statistics associated to patterns of length $2$ as follows:
  \begin{align*}
    f(\m{11}) & = \m{FF},\\
    f(\m{21}) & = \m{UU} = \m{DD},\\
    f(\m{12}) & = \m{DD(UD)^*UU} + \m{DD(UD)^*D} + \m{DD(UD)^*F} +\\
    & \;\;\;+ \m{F(UD)^+F} + \m{F(UD)^*UU} \\
    & = \mc n - \mc 1-\m{UU}-\m{FF}.
  \end{align*}
\end{thm}

\begin{proof}
  By Remark~\ref{uni}, any occurrence of the pattern $11$ in a faro
  word $w$ is formed by two consecutive singletons $xx$. From the
  definition of the bijection $f$, it follows that the number of
  occurrences of $11$ in $w$ equals the number of occurrences of $FF$
  in $f(w)$, that is $f(\m{11})=\m{FF}$.
  
  An occurrence of the pattern $21$ in $w$ is necessarily a pair in
  the decomposition of $w$. Since the length of a maximal run of
  consecutive up steps is equal to one plus the number of pairs $yx$
  in $w$ for a given $x\in[1,n]$, the number of occurrences of $21$ in
  $w$ equals the number of occurrences of $UU$ in $f(w)$. On the other hand, any nonempty dispersed Dyck path $P$ is of the form either $P=FR$ or $P=UQDR$ where $Q$ is a Dyck path and $R$ a dispersed Dyck path. Reasoning by induction, we obtain that the number  of occurrences of $DD$ equals those of $UU$ in any dispersed Dyck path, which implies $f(\m{21})=\m{UU}=\m{DD}$.

  Now, let us prove the equation $f(\m{12}) = \m{DD(UD)^*UU} + \m{DD(UD)^*D} +
  \m{DD(UD)^*F} + \m{F(UD)^+F} + \m{F(UD)^*UU}$.  An
  occurrence $xy$ of the pattern $12$ occurs in $w$ as a subblock of one
  of the following:
\begin{enumerate}
\item[(i)] two distinct consecutive pairs $(ax)(yb)$,  
\item[(ii)] two equal consecutive  pairs $(yx)(yx)$, 
\item[(iii)] a pair followed by a singleton $(ax)(y)$, 
\item[(iv)] a singleton followed by a pair $(x)(ya)$, 
\item[(v)] two distinct singletons $(x)(y)$.
\end{enumerate}

\noindent For the case (i), we distinguish three subcases.

\emph{Subcase 1.}  The occurrence $xy$ appears in a factor of the form
$(ax)(yb)$ with $b\geq a$.
This implies that neither a singleton
$s \in [a,b]$ nor a pair $pq$ with $p \in (a,b]$ or $q \in [a,b)$ can
appear in $w$. Therefore, $T_{3(s-1)}=0$ for $s \in [a,b]$,
$T_{3(p-1)-1}=1$ for $p \in (a,b]$ and $T_{3(q-1)+1}=1$ for
any $q \in [a,b)$. Thus, between the run of $D$-steps associated
to $T_{3(a-1)-1} \ge 2$ and the run of $U$-steps associated to
$T_{3(b-1)+1} \ge 2$, there are no level steps, and the runs
of $D$-steps and $U$-steps are of length one, which creates
$m = b - a\geq 0$ peaks $UD$.  Hence, the occurrence $xy$ is
associated to an occurrence of the pattern $DD(UD)^* UU$.

\emph{Subcase 2.} The occurrence $xy$ appears in a factor of
the form $(ax)(yb)$ with $b<a$ and $a<y$. This implies that
neither a singleton $x \in [a,y)$ nor a pair $pq$ with
$p \in (a,y)$ or $q \in [a,y)$ can appear in the word $w$.
Therefore, $T_{3(x-1)}=0$ for $x \in [a,y[$, $T_{3(p-1)-1}=1$
for $p \in (a,y)$ and $T_{3(q-1)+1}=1$ for any $q \in [a,y)$.
Thus, between the run of $D$-steps associated to $T_{3(a-1)-1} \ge 2$
and the run of $D$-steps associated to $T_{3(y-1)+1} \ge 2$,
there are no level steps, and the runs of $D$-steps and $U$-steps
are of length one, which creates $m = y - a > 0$ peaks $UD$.
Hence, the occurrence $xy$ is associated to an occurrence
of the pattern $DD(UD)^+ D$.
   
\emph{Subcase 3.}
The occurrence $xy$ appears in a factor of the form
$(ax)(yb)$ with $b<a$ and $a \geq y$. By definition of a faro word, we
necessarily have $a \leq y$. Thus. we deduce $ a = y$.  So, we have
$T_{3(a-1)-1} \ge 3$, which counts all consecutive pairs $az, a>z$ in
$w$. Due to Remark~\ref{uni}, all these pairs appear consecutively in
$w$.  Thus, the number of occurrences of the form $(ax)(ab)$, for
$x,b$ such that $x \le b < a$ is equal to the number of
$DDD=DD(UD)^0D$ patterns in the $(a-1)$-th run of $D$-steps in the
corresponding dispersed Dyck path. Combining to the subcase 2, the occurrence 
$xy$ is associated to an occurrence of the pattern $DD(UD)^* D$.

In the case (ii), we have a factor of the form $(ax)(yb)$ with $a = y$
and $x=b$ and the argument from Subcase 3 of case (i) applies.
 For the remaining cases, (iii) through (v), the
occurrence $xy$ of the pattern $12$ is either created by a pair
followed by a singleton $(ax)(y)$, or by a singleton followed by a
pair $(x)(ya)$, or by two different singletons $(x)(y)$.  Arguments
similar to the ones given above, allow us to prove that an occurrence
$xy$ in $w$ corresponds to an occurrence of:
\begin{itemize}
\item $DD(UD)^{y-a}F$ for the case $(ax)(y)$,
\item $F(UD)^{a-x}UU$ for the case $(x)(ya)$, and 
\item $F(UD)^{y-x}F$ for the case $(x)(y)$.
\end{itemize}

Finally, in any $n$-length word we have  $n-1$ occurrences of $2$-length
patterns, thus $\mc n - \mc 1 = \m {11} + \m{21} + \m{12}$.
Applying the bijection $f$ to both parts of the equation, we obtain $f(\m{12}) = \mc n - \mc 1 - f(\m{11}) - f(\m{21}) = \mc n
- \mc 1 - \m{UU} - \m{FF}$.
\end{proof}

\begin{thm} \label{farogenfunc}
For $p\in\{11,12,21\}$, the trivariate generating functions $F_p(x,y,z)$ where the coefficient at $x^ny^kz^t$ is the number of $k$-ary faro words of length $n$ containing exactly $t$ occurrences of the pattern $p$ are:
    {\footnotesize \begin{align*}
        F_{11}(x,y,z)&={\frac {2y \left( xz-x-1 \right) }{-xyz+xy+{x}^{3}z-{x}^{3}+y-{x}^{2}+xz+x-1+(xz-x-1)A_1}},\\
        F_{21}(x,y,z)&={\frac {2y}{-y+{x}^{2}z-2\,x+1+A_2}},\\
        F_{12}(x,y,z)&={\frac {y \left( {x}^{3}{z}^{2}-{x}^{3}z+{x}^{2}z+xyz-xy-3\,xz+x+y-1+(xz-x+1)A_2 \right) }{ \left( {
              x}^{3}{z}^{2}-{x}^{3}z+{x}^{2}z-xyz+xy-xz-x-y+1+(xz-x+1)A_2 \right)  \left( -1+y \right) z}}+\frac{y}{1-y},
  \end{align*}}
where {\footnotesize 
$A_1=\sqrt{{x}^{4}-2\,{x}^{2}y-2\,{x}^{2}+{y}^{2}-2\,y+1}$} and {\footnotesize 
$A_2=\sqrt{{x}^{4}{z}^{2}-2\,{x}^{2}yz-2\,{x}^{2}z+{y}^{2}-2\,y+1}.$}
\end{thm}

\begin{proof}
We have $f(\s_{n,k})=\B_{n+2(k-1),k-1}$. Thus, for any pattern $p$, the trivariate generating function $F_p(x,y,z)$ is given by  $y\cdot B_p(x,\frac{y}{x^2},z)$ where $B_p(x,y,z)$ is the trivariate generating function whose coefficient at $x^ny^kz^t$ is equal to the number  of dispersed Dyck paths $P\in\mathcal{B}_{n,k}$ such that $\m q(P)=t$, where $\m q=f(\m p)$.
    
For $p=21$, Theorem \ref{faropat} has $f(\m{21})=\m{UU}$. Therefore, we decompose the set $\D$ of Dyck paths as follows:
$$
\D = \epsilon \uplus UD\D \uplus U(\D\setminus\epsilon) D \D.
$$ 
We also decompose the set $\B$ of dispersed paths as follows:
$$
\B = \epsilon \uplus F\B \uplus UD\B \uplus U(\D\setminus\epsilon)D\B.
$$
If $D(x,y,z)$ is the generating function where $x^ny^kz^t$ is the number of Dyck paths of length $n$ with $k$ peaks and $t$ occurrences of $UU$, then the above algebraic equation yields $D(x,y,z)=1+x^2yD(x,y,z)+x^2z(D(x,y,z)-1)D(x,y,z)$. If $B_{21}(x,y,z)$ is the generating function whose coefficient at $x^ny^kz^t$ is the number of dispersed Dyck paths of length $n$ with $k$ peaks and $t$ occurrences of $UU$, then the above decomposition of $\B$ yields the functional equation
$$
B_{21}(x,y,z) = 1 + xB_{21}(x,y,z) + x^2yB_{21}(x,y,z) + x^2z(D(x,y,z)-1)B_{21}(x,y,z),
$$
which, in turn, yields the desired result.

For $p=11$, Theorem \ref{faropat} has $f(\m{11})=\m{FF}$. Therefore, we decompose the set $\D$ of Dyck paths as follows:
$$
\D=\epsilon \uplus UD\D\uplus U(\D\setminus\epsilon) D \D.
$$ 
We also decompose the set $\B$ of dispersed Dyck paths as follows:
$$
\B=\epsilon \uplus \F \uplus \F UD \B \uplus \F U(\D\setminus\epsilon)D\B \uplus UD\B \uplus U(\D\setminus\epsilon)D\B,
$$
where $\F$ is the infinite set of paths $F^k$ for
$k\geq 1$.  Denote by $F(x,y,z)$ the generating function for $\F$,
where its coefficient at $x^ny^kz^t$ is the number of $n$-length paths
from $\F$ having $k$ peaks and $t$ occurrences of a pattern $FF$.
Notice that $F(x,y,z)=\frac{x}{1-xz}$.  If $D(x,y)$ is the generating
function where the coefficient at $x^ny^k$ is the number of Dyck paths
of length $n$ with $k$ peaks, then the above set decomposition yields
$D(x,y)=1+x^2yD(x,y)+x^2(D(x,y)-1)D(x,y)$. Using the second set
decomposition of $\mathcal{B}$, we obtain a functional equation
\begin{align*}
B_{11}(x,y,z)&= 1+x(1+z(F(x,y,z)-1))+x^3(1+z(F(x,y,z)-1))B_{11}(x,y,z)\\
&\phantom{=} \, +x^3(1+z(F(x,y,z)-1))D(x,y)B_{11}(x,y,z)+x^2yB_{11}(x,y,z) \\
&\phantom{=} \, +x^2(D(x,y)-1)B_{11}(x,y,z),
\end{align*}
which provides the result.

For $p=12$, we have, for any $P\in \mathcal{B}_{n,k}$,
that $\m{12}(P) = n-1-\m{11}(P)-\m{21}(P)$
(that is $\m{12}= \mc n- \mc 1-\m{11}-\m{21}$),
and thus
$$
F_{12}(x,y,z) = \frac{1}{z} \left(F_{11+21} \left( xz,y,\frac{1}{z} \right) - \frac{y}{1-y} \right)+\frac{y}{1-y}.
$$
According to Theorem \ref{faropat}, we have
$f(\m{11}+\m{21})=\m{FF}+\m{UU}$. Therefore, we decompose the set $\B$
as before for the case of pattern $11$, and construct a functional
equation by taking into account the different occurrences of $FF$ and
$UU$, which yields the claimed result.
\end{proof}

\begin{cor}
    \label{faropop}
        For $n\geq 0$, the popularity  of pattern $p\in\{11,12,21\}$ in
        $\s_{n,k}$ is given by the bivariate generating
        function $G_p(x,y)$:
        \begin{align*}
          G_{11}(x,y)&={\frac {4{x}^{2}y}{ \left( 1-y-2x+x^2+A_1 \right) ^{2}}},\\
          G_{21}(x,y)&={\frac {2{x}^{2}y \left( 1+y-{x}^{2}-A_1 \right) }{ \left( 
          1-y-2x+x^2+A_1 \right)^{2}A_1}},\\
          G_{12}(x,y)&={\frac {2xy \left( A_3+(x^3-2x^2+2xy-2x-2y+2)A_1 \right) }{ \left( 1-y-2x+x^2+A_1 \right)^{2} \left( 1-y \right) A_1}},\\
        \end{align*}
        where $A_1=\sqrt {{x}^{4}-2\,{x}^{2}y-2\,{x}^{2}+{y}^{2}-2\,y+1}$ and
        $A_3={x}^{5}-2\,{x}^{4}-{x}^{3}y-3\,{x}^{3}+4\,{x}^{2}y+4\,{x}^{2}-2\,xy-2\,{y}^{2}+2\,x+4\,y-2$.
\end{cor}

\begin{proof} Using Theorem \ref{farogenfunc}, we obtain the result by calculating  $\left(\frac{\partial}{\partial z} F_p(x,y,z)\right)\big\vert_{z=1}$ for $p\in\{11,21,12\}$.
\end{proof}

\begin{cor}
  \label{faroavoid}
  For $p\in\{11,12,21\}$, the bivariate generating functions
  $H_p(x,y)$ whose coefficient at $x^ny^k$ is the number of $k$-ary faro words of length $n$ avoiding the pattern $p$ are:
        
  \begin{align*}
    H_{11}(x,y) & = {\frac {2y \left(x+1 \right) }{1-x-y-xy+{x}^{2}+{x}^{3}+(1+x)\sqrt {(x^2-2x-y+1)(x^2+2x-y+1)}}},\\
    H_{21}(x,y) & = \frac{y}{1-x-y}, \\
     H_{12}(x,y) & = {\frac {y \left( -{x}^{3}y+{x}^{2}y-x{y}^{2}+xy+{y}^{2}-2y+1 \right) }{x{y}^{3}-3x{y}^{2}-{y}^{3}+3xy+3{y}^{2}-x-3y+1}}.
\end{align*}
  \end{cor}

\begin{proof} 
Note that $H_p(x,y)=F_p(x,y,0)$, where $F_p(x,y,z)$ is as in Theorem \ref{farogenfunc}.
\end{proof}

Now we discuss the two special cases of $k=2$ and
$k=n$, which correspond respectively to binary words and $n$-ary words of length $n$ (see Table~\ref{tabpop0} for numerical values).
    
Case $k=2$: using Corollary~\ref{faropop}, we can easily prove that the popularity of the pattern $11$ in $\s_{n,2}$ generates
a shift of the sequence~\oeis{A212964} in~\cite{Slo}, which also
counts the number of 3-element subsets $A$ of $\{1,\ldots, n+1\}$ such that all
the sums $a_1+a_2$ with $a_1\leq a_2$ and $a_1,a_2\in A$ are
distinct. The popularity of $21$ generates a shift of the
sequence~\oeis{A006918} where the general term is given by
$\binom{n+3}{3}/4$ if $n$ is odd, and $n(n+2)(n+4)/24$ if $n$ is
even. The other patterns do not provide known sequences in
~\cite{Slo}.

Case $k=n$: the sequences of popularity of $p\in\{11,21,12\}$ are not listed in ~\cite{Slo}, and we have not succeeded in finding a closed form for the diagonal of $G_p(x,y)$. However, using the  Maple package \emph{gfun}~\cite{gfun}, we conjecture that the popularity sequence for $11$ satisfies a recurrence equation
$Q_1(n)u_n+Q_2(n)u_{n+1}+Q_3(n)u_{n+2}+Q_4(n)u_{n+3}=0$, where $Q_1,Q_2,Q_3,Q_4$ are some polynomial functions of degree at most $10$, which suggests that the generating function of the diagonal is D-finite when $p=11$. However, we have not succeeded in obtaining a closed form of the diagonal of $H_{11}(x,y)$. In contrast, a simple
study of the residues (see \cite{sta} Section 6.3) of $H_{21}(x/y,y)$ at the pole $y_0=(1-\sqrt{1-4x})/2$ yields the generating function $(1-\sqrt{1-4x})/(2\sqrt{1-4x})$ of the diagonal of $H_{21}(x,y)$, and its general term is, therefore, $\binom{2n-1}{n}$ (see sequence \oeis{A001700}). A similar study for the pattern $12$ yields the diagonal $x(x^3-2x^2+x+1)/(1-x)^2$ (here, the pole is $y_0=x$), which generates the sequence $u_1=1$, $u_2=3$, $u_n=n$ for
$n\geq 3$.
\begin{table}[ht!]
    \centering
   $ \begin{array}{c|c|l}
      k&\mbox{Pattern $p$} & \mbox{Popularity of $p$ in $\mathcal{S}_{n,k}$ for $1\leq n\leq 9$}   \\\hline\hline
     2 &11  &     0, 2, 6, 14, 26, 44, 68, 100, 140, \dots\\
     & 21  & 0, 1, 2, 5, 8, 14, 20, 30, 40,  \dots      \\
      &12&  0,1,4,8,14,22,32,45,60, \dots   \\
      \hline
      n &11  & 0,2,12,80,490, 3192, 20076,13094,83655, \dots\\
     & 21  & 0,1,8,85, 574,4788,31800, 24489,162305,\dots     \\
      &12&  0,1,16,135,1036, 7700,53964,38646,2636920,\dots \\
      \hline
    \end{array}$
    \caption{
      Popularity of patterns $p$ of length two in $\mathcal{S}_{n,2}$ and $\mathcal{S}_{n,n}$.}
    \label{tabpop0}
 \end{table}
\medskip

Statistic correspondences for other patterns can be obtained using a method similar to that of Theorem \ref{faropat}. Therefore, we list directly (without proof) in Theorem \ref{faropatbis} the $f$-images of all statistics associated to a pattern of length three.  It is worth noting that the reverse-complement $\chi$ is a bijection on $\mathcal{S}_{n,k}$, which proves that the statistics $\m{112}$ and $\m{122}$ (resp. $\m{121}$ and $\m{212}$, resp.  $\m{132}$ and $\m{213}$) have the same distribution on $\mathcal{S}_{n,k}$.  

\begin{thm}
  \label{faropatbis}
  For $n\geq 0$, the bijection $f$ from $\s_{n,k}$ to  $\B_{n+2(k-1),k-1}$ translates statistics associated to patterns of length three as follows: 
  \begingroup
  \allowdisplaybreaks
  \begin{align*}
    f(\m{111}) & = \m{FFF},\\
    f(\m{112}) & = \m{FF(UD)^+F} + \m{FF(UD)^*UU} ,\\
    f(\m{122}) & = \m{F(UD)^+FF} + \m{DD(UD)^*FF} ,\\
    f(\m{121}) & = \m{FUU} + \m{UUU} ,\\
    f(\m{212}) & = \m{DDF} + \m{DDD} ,\\
    f(\m{132}) & = \m{F(UD)^+UU} + \m{U(UD)^+UU} + \m{DD(UD)^*UU} ,\\
    f(\m{213}) & = \m{DD(UD)^+F} + \m{DD(UD)^+D} + \m{DD(UD)^*UU} ,\\
    f(\m{123}) & = \m{DD(UD)^*F(UD)^*UU} + \m{DD(UD)^*F(UD)^+F} \\
    &\phantom{=} \; + \m{F(UD)^+F(UD)^*UU} + \m{F(UD)^+F(UD)^+F},\\
    f(\m{211})&=f(\m{221})=f(\m{231}) = f(\m{312}) = f(\m{321}) =\mc{0}.\\
  \end{align*}
  \endgroup
\end{thm}

It would be interesting to see how the method developed
in~\cite{Band2, Band1} could be applied to obtain more pattern distributions in dispersed Dyck paths, but this is beyond the scope of the present paper. The multivariate generating functions for patterns of length three are quite technical, not particularly interesting and laborious to obtain. So, we decide to leave them as an exercise for the reader. 


\section{Patterns in faro permutations} \label{sec:pat_faro_perm}

We say that a $k$-ary faro word $w$ of length $n$ is \emph{injective}
(resp. \emph{surjective}) if and only if any value in $w$ appears only
once in $w$ (resp. any value $x\in[1,k]$ appears in $w$). A \emph{faro
permutation} of length $n$ is an $n$-ary faro word that is both
injective and surjective. Let $\p_n$ be the set of length $n$ faro
permutations. For instance, we have $\p_3=\{123,132,213\}$. Since faro
permutations are entirely determined by the choice of their values on
the odd indices, the cardinality of $\p_n$ is $\binom{n}{\left\lfloor
  n/2\right\rfloor}$.
Note that  faro permutations are permutations avoiding the three consecutive patterns $231$, $321$ and
$312$ (see Remark~\ref{excpat}).

\begin{thm}
  The bijection $f$ maps surjective $k$-ary faro
  words of length $n$ onto dispersed Dyck paths in $\B_{n+2(k-1),k-1}$ avoiding $UDUD$ that neither start nor end with $UD$.
\end{thm}
  
\begin{proof}
  Using the definition of the bijection $f$ and in particular the definition of the sequence $T$, surjective faro words are those that have a sequence $T$ satisfying (i) $T_0+T_1>1$, (ii)
  $T_{3(x-1)-1}+T_{3(x-1)}+T_{3(x-1)+1}>2$
  for  $x \in [2, k-1]$, and (iii)
  $T_{3(k-1)-1}+T_{3(k-1)}>1$. Since $T_1\geq 1$, the condition (i) is equivalent to $T_0\neq 0$, or $T_0=0$ and $T_1>1$, which means that $f(w)$ does not start with $UD$. Similarly, the condition (iii) is equivalent to the fact that $f(w)$ does not end with $UD$. Since $T_{3(x-1)-1}\geq 1$ and $T_{3(x-1)+1}\geq 1$, the condition (ii) is equivalent to $T_{3(x-1)}>0$, or $T_{3(x-1)}=0$ and $T_{3(x-1)-1}+T_{3(x-1)+1}>2$, which means that $f(w)$ does not contain any occurrence of $UDUD$.
\end{proof}
  
\begin{thm}
  The bijection $f$ maps injective $k$-ary faro words
  of length $n$ into dispersed Dyck paths in $\B_{n+2(k-1),k-1}$ avoiding the patterns $FF$, $DDD$, $UUU$, $DDF$, $FUU$, and $DDUU$.
\end{thm}

\begin{proof}
   Using the definition of $f$,
  injective faro words are those that have a sequence $T$ satisfying
  (i) $T_{3(x-1)}<2$, $x\in[1,k]$; (ii) $T_{3(x-1)-1}<3$, $x\in[2,k]$;
  (iii) $T_{3(x-1)+1}<3$, $x\in[1,k-1]$;
  (iv) $T_{3(x-1)-1}+T_{3(x-1)}<3$, $x\in[2,k]$;
  (v)  $T_{3(x-1)}+T_{3(x-1)+1}<3$, for $x\in[1,k-1]$; and
  (vi) $T_{3(x-1)-1}+T_{3(x-1)+1}<3$, for $x\in[2,k-1]$.
  It means that $f(w)$ avoids,
  respectively, the patterns $FF$, $DDD$, $UUU$, $DDF$, $FUU$ and $DDUU$.
\end{proof}

\begin{thm}
  The image by $f$ of $\p_n$ is the subset $B'_{3n-2,n-1}$ of
  dispersed Dyck paths in $B_{3n-2,n-1}$ that neither start nor end
  with $UD$ and where any two consecutive occurrences of $UD$ are
  separated by exactly one step.
\end{thm}

\begin{proof}
  The two previous theorems imply that $f(\p_n)$ is the set of
  dispersed Dyck paths in $B_{3n-2,n-1}$ that neither start nor end
  with $UD$ and that avoid the patterns $FF$, $DDD$, $UUU$, $DDUU$,
  $UDUD$, $DDF$ and $FUU$, which is exactly the dispersed Dyck paths
  that neither start nor end with $UD$ and where any two consecutive
  occurrences of $UD$ are separated by exactly one step. Indeed, in any dispersed Dyck path, the subpath between two consecutive occurrences of $UD$ is necessarily of the form $D^iF^jU^k$ for
  $i,j,k\geq 0$. Then, the avoidance of $FF$, $DDD$ and $UUU$ implies that $i,j,k\leq 1$, and the avoidance of the other patterns implies that $i+j+k=1$ as claimed.    
\end{proof}
    
  Thus, we deduce a one-to-one correspondence $g$
  between length $n$ faro permutations and dispersed Dyck paths of length $n$, where $g(p)$ is obtained from $p\in\p_n$ by removing all occurrences of $UD$ in $f(p)$. For instance, if $p=1243576$ then $f(p)=FUDFUDUUDDUDFUDUUDD$ and $g(p)=FFUDFUD$.  

\begin{thm}
  \label{permbij}
  For $n\geq 0$, the bijection $g$ from $\p_n$  to 
  $\B_n$ transports the pattern statistics as follows:
  \begingroup
  \allowdisplaybreaks
  \begin{align*}
    g(\m{21}) & = \m{U},\\
    g(\m{12}) & = \m{DU} + \m{DD} + \m{DF} + \m{FF} + \m{FU},\\
              & = \mc n - \mc 1 - \m{U},\\
    g(\m{132}) & = \m{FU} + \m{UU} + \m{DU},\\
    g(\m{213}) & = \m{DF} + \m{DD} + \m{DU}, \\
    g(\m{123}) & = \m{DFU} + \m{DFF} + \m{FFU} + \m{FFF},\\
               & = \mc n - \mc 2 - \m{FU}-\m{UU}-2~\m{DU}-\m{DF}-\m{DD},\\
     g(\m{231}) &= g(\m{312}) = g(\m{321}) =\mc {0}.\\
  \end{align*}
  \endgroup
\end{thm}

\begin{proof}
   For a faro permutation $w$,
   $g(w)$ is obtained from $f(w)$ by removing all peaks $UD$.
   In $f(w)$ consecutive occurrences of $UD$ are separated by one letter exactly, and no $U^k$, $D^k$ or $F^j$ exists for $k \ge 3, j \ge 2$.
   The last $U$ in $U^2$ and the first $D$ in $D^2$ must be a part of an occurrence of $UD$.
   The claimed statistic equations are obtained from Theorem~\ref{faropat}
  and Theorem~\ref{faropatbis} by deleting peaks $UD$ and replacing all remaining 
   $U^2$ (resp. $D^2$) with $U$ (resp. $D$) in all considered pattern statistics.
   Only two (resp. three) patterns of length
  2 (resp. 3) are possible in a faro permutation, namely $12$ and $21$
  (resp. $123$, $132$ and $213$).  In any $n$-length word there is
  $n-1$ (resp. $n-2$) occurrences of patterns of length $2$
  (resp. $3$). It follows that $g(\m{12}) = \mc n - \mc 1 - \m{U}$ and
  $g(\m{123}) = \mc n - \mc 2 - \m{FU}-\m{UU}-2~\m{DU}-\m{DF}-\m{DD}$.
  
\end{proof}

\begin{thm}
  \label{faropermgenfunc}
For $p\in\{21,12,132,213,123\}$, the bivariate generating functions
$K_p(x,y,z)$, where the coefficient at $x^ny^k$ is the number of faro
permutations of length $n$ containing exactly $k$ occurrences of the
pattern $p$, are:
\begingroup \allowdisplaybreaks
     \begin{align*}
        K_{21}(x,y)  & = \frac{2}{1-2x+\sqrt {1-4\,{x}^{2}y} }, \\
        K_{12}(x,y)  & = {\frac {1+y+2xy-2\,x{y}^{2}+(y-1)\sqrt {1-4\,{x}^{2}y}}{y \left(1 -2\,xy+\sqrt {1-4\,{x}^{2}y} \right) }}, \\
        K_{132}(x,y) & = {\frac {1+y+(y-1)\sqrt {1-4\,{x}^{2}y}}{y \left(1-2x+\sqrt {1-4\,{x}^{2}y} \right) }}, \\
        K_{213}(x,y) & = K_{132}(x,y), \\
        K_{123}(x,y) & = \frac{2+3x-3xy+2x^2-2 x^{2} y  -x(1-y) \sqrt{1 - 4 x^{2}}  }{1-2 x y + \sqrt{1 - 4 x^{2}} }.       
    \end{align*}
\endgroup
  \end{thm}

\begin{proof}
For $p=21$, Theorem~\ref{permbij} has $g(\m{21})=\m{U}$. So, we decompose the set of Dyck paths as $\D = \epsilon \uplus U \D D \D$, the set of dispersed Dyck paths as $\B = \epsilon \uplus F\B \uplus U \D  D \B$, and obtain the following system:
$$
  \begin{cases}
    D(x,y) = 1 + x^2 y D^2(x,y),\\
    B = 1 + x B(x,y) + x^2 y D(x,y) B(x,y),
  \end{cases}
$$
where $D(x,y)$ (resp. $B(x,y)$) is the generating function for the set of Dyck paths (resp. dispersed Dyck paths) with respect to the number of occurrences of $U$. Solving it, we obtain $K_{21}(x,y) = B(x,y)$. 

Since only the two length 2 patterns ($12$ and $21$) are possible
in a faro permutation, we have $\m {12} = \mc n - \mc 1 - \m
{21}$. Hence, $K_{12} (x,y) = (K_{21} (xy, \frac{1}{y}) - 1)/y + 1$.

Only tree patterns of length 3 are possible in a faro permutation,
$123$, $132$ and $213$, so we have $ \m{123} = \mc n - \mc 2 - \m{132}
- \m{213}$. By Theorem~\ref{permbij}, $g(\m{132} + \m{213}) = \m{FU} +
\m {UU} + 2~\m {DU} + \m {DF} + \m {DD}$.  We decompose the sets of
Dyck and dispersed Dyck paths as follows: 
\[
\begin{cases}
      \D &= \epsilon \uplus UD \uplus U (\D \setminus \epsilon) D \uplus UD (\D \setminus \epsilon ) \uplus U (\D \setminus \epsilon) D (\D \setminus \epsilon) ,\\
      \B &= \epsilon\uplus \overline\B \uplus
      U (\D \setminus \epsilon) D \uplus
      U (\D \setminus \epsilon) D \overline\B \uplus
      U (\D \setminus \epsilon) D (\B \setminus (\epsilon \uplus \overline\B)) \\
      &\phantom{=} \uplus \,
      U D \uplus
      U D \overline\B \uplus
      U D (\B \setminus (\epsilon \uplus \bar\B)), \\
      \overline\B &= \F \uplus \F (\B \setminus (\epsilon \uplus \overline\B)),
\end{cases}
\]
where $\F$ is the set of paths $F^k$, $k \ge 1$ and $\overline\B$ is the set of dispersed Dyck paths starting with a level step. From this decomposition we obtain the following system of functional equations:
\[
\begin{cases}
      D(x, y) & = 1 + x^2 + 2 x^2 y^2 \left(D(x,y) - 1\right) +  x^2 y^4 \left(D(x,y) - 1\right)^2, \\
      B(x, y) & = 1 + \overline B(x,y)
      + x^2y^2\left( D(x,y) - 1\right)
      + x^2y^3\left( D(x,y) - 1\right) \overline B(x,y) + \\
      &\phantom{=} + x^2y^4\left( D(x,y) - 1\right) \left(B(x,y) - \overline B(x,y) - 1\right)
      + x^2
      + x^2 y \overline B(x,y) \\
      &\phantom{=} + x^2 y^2 \left(B(x,y) - \overline B(x,y) - 1\right),
      \\
      \overline B(x, y) & = \frac{x}{1-x} + \frac{xy}{1-x}\left( B(x,y) - \overline B(x,y) - 1 \right),
\end{cases}
\]
    where $D(x,y)$ (resp. $B(x,y)$, resp. $\overline B(x,y)$) is the
    generating function for the set of Dyck paths (resp. dispersed
    Dyck paths, resp. dispersed Dyck paths starting with $F$) with
    respect to the statistics $\m{FU} + \m {UU} + 2\,\m {DU} + \m {DF}
    + \m {DD}$.  After solving this system, we obtain the result by
    evaluating $K_{123} (x,y) = 1 + x + (B(xy,\frac{1}{y}) - 1 - xy) /
    y^2 $.
    Note that it is possible to look directly at $g(\m{123})
    = \m{DFU} + \m{DFF} + \m{FFU} + \m{FFF}$ rather than at $g(\m{132}
    + \m{213})$ as we did, but the decomposition will be more
    complicated.
    
    Note that $K_{132}(x,y) = K_{213}(x,y)$, by taking the
    reverse-complement of faro permutations.  We remark that $\m{DF}+\m{DD}+\m{DU}$ corresponds to the number of occurrences of $D$ except the last symbol if it is a $D$. So, if $\mathcal{D}$ is the set of Dyck paths and $\mathcal{B}$  the set of dispersed paths, then the classical decompositions $\mathcal{D}=\epsilon\uplus\mathcal{D}U\mathcal{D}D$, $\mathcal{B}=\epsilon\uplus\mathcal{B}F\uplus\mathcal{B}U\mathcal{D}D$ provide the following system of functional equations:
    \[
    \begin{cases}
       D(x,y)& = 1+x^2yD(x,y)^2,\\
       B(x,y)& = 1+xB(x,y)+x^2yB(x,y)D(x,y),\\
       \overline B(x,y)&=1+xB(x,y)+x^2B(x,y)D(x,y),
    \end{cases}
    \]
%
    where $D(x,y)$ and $B(x,y)$ (resp. $\overline B(x,y)$) are the generating functions for the sets of Dyck paths and dispersed Dyck paths with respect to the number of $D$ (resp. number of $D$ except the last symbol if it is a $D$).
    Solving the system, we obtain $K_{132}(x,y) = K_{213}(x,y) = \overline B(x,y)$.
\end{proof}

\begin{cor}
    \label{faropermpop}
        For $n\geq 0$, the popularity of pattern $p\in\{21,12,132,213,123\}$ in
        $\p_{n}$ is given by the generating
        function $L_p(x)$:
        \begin{align*}
         L_{21}(x)  & = \frac{1 - \sqrt{1-4x^2}}{ 2(1-2x) \sqrt{1 - 4 x^{2}} },\\
          L_{12}(x)  & = {\frac {2x \left( -1+4\,{x}^{2}+x+\sqrt {1-4\,{x}^{2}}\right) }{(1-2x)(1 + \sqrt{1 - 4 x^{2}}) \sqrt {1-4\,{x}^{2}} }},\\
          L_{132}(x) & = {\frac {x \left(-1 +4{x}^{2}+ 2x + (1-2x)\sqrt {1-4{x}^{2}} \right) }{ (1-2x)(1 + \sqrt{1 - 4 x^{2}}) \sqrt {1-4{x}^{2}}}},\\
          L_{213}(x) & = L_{132}(x),\\
          L_{123}(x) & = \frac{x (1 + 2 x)(1-\sqrt{1 - 4 x^{2}})}{  (1-2 x)(1+\sqrt{1 - 4 x^{2}}) }.
        \end{align*}
\end{cor}

\begin{proof} We evaluate  $\frac{\partial K_p(x,y)}{\partial y } \big\vert_{y=1}$.
\end{proof}

\begin{cor}
    \label{faropermpopclose} 
    For $n\geq 2$, the popularity  of pattern $p\in\{21,12,132,213,123\}$ in
        $\p_{n}$ is given by $\m p(\p_n)$:
         \begin{align*}
         \m{21}(\p_n)&=\frac{n+1}{2} \binom{n}{\lfloor\frac{n}{2}\rfloor }
-{2}^{n-1}\sim ~ \sqrt{\frac{n}{2\pi}}\cdot 2^{n},\\
         \m{12}(\p_n)&=(n-1)\binom{n}{\lfloor\frac{n}{2}\rfloor}-\m{21}(\p_n)\sim ~ \sqrt{\frac{n}{2\pi}}\cdot 2^{n},\\
         \m{123}(\p_n)&={2}^{n}-2\,\binom{n-1}{\lfloor\frac{n-1}{2} \rfloor }-\binom{n}{\lfloor\frac{n}{2} \rfloor }\sim ~2^n,
\\
         \m{132}(\p_n)&=\m{213}(\p_n)=\frac{1}{2}\left((n-2)\binom{n}{\lfloor\frac{n}{2}\rfloor}-\m{123}(\p_n)\right)\sim ~\sqrt{\frac{n}{2\pi}}\cdot 2^{n}.
         \end{align*}
\end{cor}
\begin{proof} Recall that 
$|\p_n|={n\choose\lfloor n/2\rfloor}$. Then  we have
$\m{21}(\p_n)+\m{12}(\p_n)=
(n-1)\cdot{n\choose\lfloor n/2 \rfloor}$ and 
$\m{132}(\p_n)+\m{213}(\p_n)+\m{123}(\p_n)= 
(n-2)\cdot{n\choose\lfloor n/2 \rfloor}$, and 
considering  $\m{132}(\p_n)=\m{213}(\p_n)$, it suffices
to prove the result for $\m{21}(\p_n)$ and
$\m{123}(\p_n)$. Due to Corollary~\ref{faropermpop}, we
have $$L_{21}(x)   = \frac{W(x)}{2x} -\frac{1}{2(1-2x)}$$
where $W(x)=\frac{x}{ (1-2x) \sqrt{1 - 4 
x^{2}} }$ is the generating function for the sequence 
\oeis{A100071} in \cite{Slo} which has the general term
$\frac{n}{2}\cdot {n-1\choose 
\lfloor (n-1)/2 \rfloor }$. This induces directly 
$\m{21}(\p_n)=\frac{n+1}{2}\cdot {n\choose 
\lfloor n/2 \rfloor }-{2}^{n-1}$.

Similarly, if we expand the numerator of $L_{123}(x)$ given in 
Corollary~\ref{faropermpop}, then we obtain four generating 
functions having the general terms respectively equal to 
$2^{n-1}-{n-1\choose \lfloor (n-1)/2 \rfloor }$,
$-{n-1\choose \lfloor (n-1)/2 \rfloor }$,
$2^{n-1}-\frac{1}{2}\cdot {n\choose \lfloor n/2 \rfloor }$ 
and $-\frac{1}{2}\cdot {n\choose \lfloor n/2 \rfloor },$ 
which implies the claimed result. Finally, asymptotics are easily
obtained using ${n\choose \lfloor n/2 \rfloor }\sim
\sqrt{\frac{2}{\pi n}}\cdot 2^n$. 
\end{proof}


Using formulae from Corollary~\ref{faropermpopclose} the following remark can be easily verified.

\begin{rem}
The expected number of the occurrences of the pattern
$21$ (respectively $12$, $132$ and $213$) in a 
randomly selected faro permutation of length $n$ is 
asymptotically equivalent to $n/2$ when $n \to 
\infty$. In contrast, the expectation of $123$ is 
asymptotically equivalent to $\sqrt{ \pi n / 2 }$ and 
thus the probability 
that a random faro permutation contains an occurrence
of $123$ at a random position approaches $0$ as $n$ grows.
\end{rem}

Table~\ref{tablepop} provides the first values of the popularity of each pattern of length at most three in faro permutations.

 \begin{table}[ht!]
    \centering
    $ \begin{array}{c|l|l}
      \mbox{Pattern $p$} & \mbox{Popularity of $p$ in $\mathcal{P}_n$ for $1\leq n\leq 11$}  & \mbox{OEIS} \\\hline\hline
      21  &     0,1,2,7,14,38,76,187,374,874,1748, \dots &\oeis{A107373}\\
      \hline
      12  & 0,1,4,11,26,62,134,303,634,1394,2872, \dots      &  \oeis{A340567}\\
      \hline
      132, 213  &  0,0,1,4,10,28,61,152,318,748,1538, \dots  & \oeis{A340568}\\
      \hline
      123  &  0,0,1,4,10,24,53,116,246,520,1082, \dots   &  \oeis{A340569}\\
      \hline
      231,312,321  &  0,0,0,0, \dots   & \\
      \hline
    \end{array}$
    \caption{
      Popularity of patterns $p$ of length at most three  in faro permutations.}
    \label{tablepop}
 \end{table}

\section{Some particular subsets of $\mathcal{P}_n$ and $\mathcal{S}_{n,k}$}

In this part, we study particular subsets of faro permutations and
faro words which are in one-to-one correspondence with other sets of
well-known combinatorial objects.  Let us recall the definition of a
{\em standard cycle notation} (s.c.n.) and Foata's first fundamental
transformation $\phi$ (see~\cite{lot}).

In the standard cycle notation (s.c.n.) of a permutation $w$ each
cycle starts with its largest element, and cycles are ordered from
left to right in increasing order of their largest elements.

Foata's first fundamental transformation (see~\cite{lot}) acts on a
permutation $w$ as follows.  Write a permutation $w$ in s.c.n., and
then cyclically rearrange every cycle so that it ends with its largest
element.  Then, reverse each cycle and delete all parentheses. For
instance, if $w=7321564$, then the s.c.n. for $w$ is
$(32)(5)(6)(741)$, after rearrangement we have $(23)(5)(6)(417)$, and
thus $\phi(w)=3256714$. If $w$ is an involution, it contains only
cycles of length one or two, rearrangement and reversion are not
needed, we directly obtain the image after dropping parentheses in
s.c.n.

\begin{thm}
  \label{faroperm2inv}
  Foata's first fundamental transformation bijectively maps the set
  $\I_n(3\mbox{-}2\mbox{-}1)$ of involutions of length $n$ avoiding
  the classical pattern $3\mbox{-}2\mbox{-}1$ onto the set $\p_n$ of
  faro permutations.
\end{thm}

\begin{proof} Let us prove that the standard cycle notation
of $w\in\I_n(3\mbox{-}2\mbox{-}1)$ cannot contain any of the following consecutive
cycles: $(x)(yz)$ with $z<x$, $(xy)(z)$ with $z<x$, $(x)(y)(z)$ with 
$z<x$, or $(xy)(zt)$ with $t<y$ or $z<x$.  Assume that $w\in\I_n(3\mbox{-}2\mbox{-}1)$ and assume
towards contradiction that the standard cycle notation (s.c.n.) of $w$
contains $(x)(yz)$ with $z<x$.
Then we have $x<y$ and
thus $z<x<y$, which means that the subsequence $yxz$ (occurring at
indices $z,x,y$) is an occurrence of $3\mbox{-}2\mbox{-}1$ in $w$, a contradiction.
Due to the definition of the s.c.n. of $w$, the case
$(x)(y)(z)$ with $z<x$, the case $(xy)(z)$ with $z<x$ and the case $(xy)(zt)$ with $z<x$ do not occur
since the cycles are arranged in increasing order of their first
elements.
If the s.c.n. of $w$ contains $(xy)(zt)$ with $t<y$, then we
have $t<y<x<z$, which implies that $w$ contains an occurrence $zxy$
(at indices $t,y,x$) of $3\mbox{-}2\mbox{-}1$, a contradiction.  Thus,
$\phi(\I_n(3\mbox{-}2\mbox{-}1))\subset \mathcal{P}_n$. Since $\phi$ is injective, and
$\I_n(3\mbox{-}2\mbox{-}1)$ is also enumerated by ballot numbers $b_n$ (see for instance
\cite{bar,sim}), we have $\phi(\I_n(3\mbox{-}2\mbox{-}1))=\mathcal{P}_n$.
\end{proof}

\begin{rem} It is known that Foata's first transformation $\phi$ maps the statistic of the number of excedances (values $w_i$ such that $w_i>i$) to the statistic $\m{21}$ (number of descents $w_i>w_{i+1}$). Therefore, the generating functions $K_{21}(x,y)$ and $L_{21}(x)$ in Corollary~\ref{faropermpop} also give the distribution
and the popularity of excedances in $\I_n(3\mbox{-}2\mbox{-}1)$.
\end{rem}
\begin{rem} We could easily check that $g(\phi(w))=\Phi(w)$ for $w\in \I_n(3\mbox{-}2\mbox{-}1)$, where $\Phi$ is a bijection  in~\cite{bar} between involutions and labeled Motzkin paths, which also is a restriction of Biane's bijection~\cite{biane}, which in turn is closely related to
Françon-Viennot bijection~\cite{fra}.
\end{rem}

\bigskip
The next theorem deals with alternating faro permutations, i.e. permutations $w$ satisfying  $w_1 > w_2 <w_3 > \cdots $. Let $\mathcal{A}_n$ be the set of alternating faro permutations of length $n$.

\begin{thm}
  \label{prm-dyck} There is a bijection between $\mathcal{A}_{2n}$ and the set of Dyck paths of length $2n$.
\end{thm}

\begin{proof} Let $w$ be a faro permutation of length $2n$.
    Then, $w$ is  alternating if and only if $w$ does not contain any singleton in its block decomposition. Due to the definition of $f$, this means that $f(w)$ does not contain any $F$-steps and thus, $g(w)$ is a Dyck path of length $2n$, and vice versa.   
\end{proof}

\begin{thm}
  \label{prm-deran}  The set $\mathcal{A}_{2n}$ is exactly the set of length $2n$ faro derangements, i.e. faro permutations with no fixed point $w_i=i$ for $i\in[1,2n]$.
\end{thm}

\begin{proof} Let $w=w_1w_2\ldots w_{2n-1}w_{2n}$ be a faro permutation of length $2n$, that is $w_i<w_{i+2}$ for $1\leq i\leq 2n-2$. Then, $w$ is  alternating if and only if $w$ does not contain any singleton in its block decomposition, or equivalently,  $w$ satisfies $w_i>w_{i+1}$ if $i$ is odd, and  $w_i<w_{i+1}$ otherwise. This is equivalent to $w_i$ is greater than $w_1,w_2,\ldots, w_{i-1}$ and $w_{i+1}$ if $i$ is odd, and $w_i$ is smaller than $w_{i-1}$, $w_{i+1}, w_{i+2}, \ldots, w_{2n}$ if $i$ is even, which means that $w_i>i$ if $i$ is odd and $w_i<i$ otherwise. Thus, we have $w_i\neq i$, and $w$ is a derangement. This last implication also is an equivalence because  it cannot occur $w_i<i$ with $i$ odd, or $w_i>i$ with $i$ even in a faro derangement $w$. 
\end{proof}
\begin{thm}
  \label{prm-fibo1} 
  Let $\mathcal{B}_n$ (resp. $\mathcal{B}'_n$) be the set of length $n$ faro permutations avoiding the classical pattern $2\mbox{-}3\mbox{-}1$  (resp. the pattern $3\mbox{-}1\mbox{-}2$), then
  
 $\bullet$ The cardinality of $\mathcal{B}_n$ is given by the Fibonacci sequence
  $f_n$ defined by $f_n=f_{n-1}+f_{n-2}$ with $f_1=1$, $f_2=2$.
  
 $\bullet$ We have $\mathcal{B}_n=\mathcal{B}'_n$.
 
 $\bullet$ $\mathcal{B}_n$ is exactly the set of length $n$ faro involutions.
\end{thm}

\begin{proof}
A faro permutation $w$ avoiding the pattern $2\mbox{-}3\mbox{-}1$ is of the form $1w'$ or $21w'$, where $w'$ also is a faro permutation avoiding $2\mbox{-}3\mbox{-}1$. Indeed, if a faro permutation $w$ starts with $x>2$, then $w$ starts with $x1y$ for some $y>x$. Then the value 2 is to the right of $x1y$, which creates an occurrence $xy2$ of $2\mbox{-}3\mbox{-}1$, a contradiction. Therefore, the
cardinality $f_n$ of length $n$ faro permutations satisfies
$f_n=f_{n-1}+f_{n-2}$ with $f_1=1$, $f_2=2$. Using the same argument, faro permutations avoiding $2\mbox{-}3\mbox{-}1$ are  also faro permutations avoiding $3\mbox{-}1\mbox{-}2$.

For the third statement, due to the decomposition of $w\in\mathcal{B}_n$ (either $w=1w'$ or $w=21w'$ with $w'\in \mathcal{B}$), we conclude by induction that $w$ is necessarily a faro involution. Conversely, a faro involution $w$ avoids the classical pattern $321$, which implies that $w=1w'$ or $w=21w'$ where $w'$ is also a faro involution (if the first entry $w_1$ of $w$ satisfies $w_1\geq 3$, then we have $w_{w_1}=1$; since $w$ avoids $321$, we necessarily have $w_i> w_1\geq 3$ for $1\leq i\leq w_1-1$, and thus $w_{w_1-2}>w_{w_1}=1$, which is not possible in a faro permutation). A simple induction implies that $w\in\mathcal{B}_n$, which completes the proof.
\end{proof}

In the following, we consider (for convenience) faro words on the $n$-ary alphabet $[0,n-1]$, and we focus on the set of \emph{subexcedent} faro words of length $n$, i.e. faro words $w_1w_2\ldots w_n$ satisfying $w_i\leq i-1$ for $1\leq i\leq n$.
We make a shift $[1,n]\rightarrow [0,n-1]$ on the alphabet in order to apply directly the  results presented in \cite{Bur}.

\begin{thm}
  \label{prm-dum}
 There is a bijection between  subexcedent  faro words of length $n$ and $2\mbox{-}1\mbox{-}4\mbox{-}3$-avoiding  Dumont permutations of the second kind
  of length $2n$.
\end{thm}   

We will briefly recall the result given in \cite{Bur} that enumerated $2\mbox{-}1\mbox{-}4\mbox{-}3$-avoiding Dumont permutations of 
the second kind of length $2n$. Dumont permutations of the second kind of length $2n$ are permutations $\pi$ that satisfy the following conditions for $i\in[n]$:
\[
\pi(2i-1)\ge 2i-1, \qquad \pi(2i)\le 2i-1.
\]

In other words, the values in the odd positions are weak excedances,
whereas the values in the even positions are deficiencies. In
addition, if $\pi$ avoids the pattern $2\mbox{-}1\mbox{-}4\mbox{-}3$
(i.e. does not contain a subsequence
$\pi(i_1)\pi(i_2)\pi(i_3)\pi(i_4)$ of length $4$ such that
$i_1<i_2<i_3<i_4$ and $\pi(i_2)<\pi(i_1)<\pi(i_4)<\pi(i_3)$), then the
values in the even positions of $\pi$ are exactly $\{1,2,\dots,n\}$,
and the values in the odd positions of $\pi$ are exactly
$\{n+1,n+2,\dots,2n\}$. Moreover, the subsequence of values of $\pi$
in the even positions avoids the pattern $2\mbox{-}1\mbox{-}3$ while the subsequence
of values of $\pi$ in the odd positions avoids the pattern $1\mbox{-}3\mbox{-}2$. This
allows~\cite{Bur} to construct a bijection as in
Krattenthaler~\cite{Kra01} from the even-position subsequence of $\pi$
to north-east integer lattice paths from $(0,0)$ to $(n,\left\lfloor
n/2 \right\rfloor)$ staying on or below the line $y=x/2$, and from the
odd-position subsequence of $\pi$ to the same paths but ending at
$(n+1,\left\lfloor (n+1)/2 \right\rfloor)$. Let $\{a_n\}_{n\ge 0}$ be
the sequence \href{https://oeis.org/A047749}{A047749} \cite{Slo}, so
that

\[
a_{2n}=\frac{1}{2n+1}\binom{3n}{n}, \qquad a_{2n+1}=\frac{1}{n+1}\binom{3n+1}{n},
\]
then the number of $2\mbox{-}1\mbox{-}4\mbox{-}3$-avoiding Dumont permutations of the second kind of length $2n$ is $a_{n}a_{n+1}$. Thus, to prove Theorem \ref{prm-dum}, we only need to construct a bijection from subexcedent faro words of length $n$ to ordered pairs of north-east lattice paths on or below the line $y=x/2$ from $(0,0)$ to $(n,\left\lfloor n/2 \right\rfloor)$ and $(n+1,\left\lfloor (n+1)/2 \right\rfloor)$, respectively.

\begin{proof}[Proof of Theorem \ref{prm-dum}] 
Let $\pi$ be a subexcedent faro word of length $n$.
As in \cite{Bur}, let $\pi_o$ and $\pi_e$ be the odd-position and even-position subsequences of $\pi$. Then $\pi_o$ and $\pi_e$ are nondecreasing subsequences such that
\begin{equation} \label{eq:subexc-even-odd}
\begin{split}
\pi_o(i)&=\pi(2i-1)\in[0,2i-2], \quad i\le \left\lfloor\frac{n+1}{2}\right\rfloor,\\
\pi_e(i)&=\pi(2i)\in[0,2i-1], \quad i\le \left\lfloor\frac{n}{2}\right\rfloor.
\end{split}
\end{equation}

Conversely, any word $\pi$ whose odd-position and even-position subsequences $\pi_o$ and $\pi_e$ satisfy the above properties is a subexcedent faro word of length $n$. 
Given sequences $\pi_o$ and $\pi_e$ as in \eqref{eq:subexc-even-odd}, associate to them a pair of north-east lattice paths as follows. If $\pi_o$ or $\pi_e$ has a letter $a_i$ in position $i$, map such an entry to the point $(i-1,a_i)$ in the integer lattice. Let $k=\left\lfloor\frac{n+1}{2}\right\rfloor$ for $\pi_o$ and $k=\left\lfloor\frac{n}{2}\right\rfloor$ for $\pi_e$, and let $a_{k+1}=2k$ for $\pi_o$ and $a_{k+1}=2k+1$ for $\pi_e$.


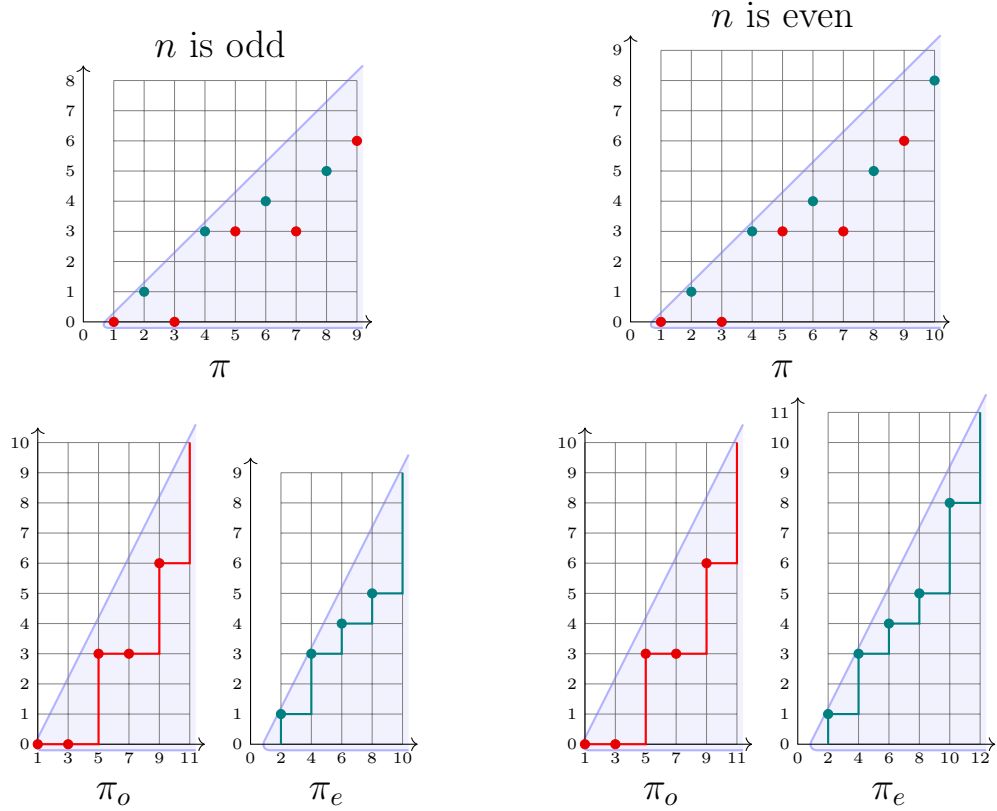
\begin{figure}[ht]
\begin{center}
{\scriptsize
\begin{tikzpicture}[scale=0.4]

\begin{scope}[xshift=1.5cm,yshift=0cm,local bounding box=odd]

\node[above] at (4.5,8.5) {\large $n$ is odd};
\node[below] at (4.5,-1) {\large $\pi$};

\draw[thick,rounded corners, blue!30,fill=blue!5] (0,0) ++(1,0)
+(8.2,8.5) -- 
+(-0.5,-0.2) -- 
+(+8.2,-0.2);

\draw[help lines] ++(1,0) +(0,0) grid +(8,8);

\draw[thin, <->] (0,8.5) |- (9.5,0);

\foreach \x in {0,1,2,3,4,5,6,7,8,9} 
{
\node[below] at (\x,0) {\tiny $\x$};
}

\foreach \x in {0,1,2,3,4,5,6,7,8} 
{
\node[left] at (0,\x) {\tiny $\x$};
}

\foreach \x/\y in {1/0,3/0,5/3,7/3,9/6}
{
\filldraw[red!90!black] (\x,\y) circle (1.5mm);
}

\foreach \x/\y in {2/1,4/3,6/4,8/5}
{
\filldraw[teal] (\x,\y) circle (1.5mm);
}

\end{scope}

\begin{scope}[xshift=0cm,yshift=-14cm,local bounding box=odd-odd]

\node[below] at (2.5,-1) {\large $\pi_o$};

\draw[thick,rounded corners, blue!30,fill=blue!5] (0,0)
+(5.2,10.6) -- 
+(-0.2,-0.2) -- 
+(+5.2,-0.2);

\draw[help lines] +(0,0) grid +(5,10);

\draw[thin, <->] (0,10.5) |- (5.5,0);

\foreach \x in {1,3,5,7,9,11} 
{
\node[below] at (\x/2-1/2,0) {\tiny $\x$};
}

\foreach \x in {0,1,2,3,4,5,6,7,8,9,10} 
{
\node[left] at (0,\x) {\tiny $\x$};
}

\draw[red,thick] (0,0)
-| (1,0)
-| (2,3)
-| (3,3)
-| (4,6)
-| (5,10);

\foreach \x/\y in {0/0,1/0,2/3,3/3,4/6}
{
\filldraw[red!90!black] (\x,\y) circle (1.5mm);
}

\end{scope}

\begin{scope}[xshift=7cm,yshift=-14cm,local bounding box=odd-even]

\node[below] at (2.5,-1) {\large $\pi_e$};

\draw[thick,rounded corners, blue!30,fill=blue!5] (0,0)
+(5.2,9.6) -- 
+(-0.2+0.5,-0.2) -- 
+(+5.2,-0.2);

\draw[help lines] ++(1,0) +(0,0) grid +(4,9);

\draw[thin, <->] (0,9.5) |- (5.5,0);

\foreach \x in {0,2,4,6,8,10} 
{
\node[below] at (\x/2,0) {\tiny $\x$};
}

\foreach \x in {0,1,2,3,4,5,6,7,8,9} 
{
\node[left] at (0,\x) {\tiny $\x$};
}

\draw[teal,thick] (1,0)
-| (1,1)
-| (2,3)
-| (3,4)
-| (4,5)
-| (5,9);

\foreach \x/\y in {1/1,2/3,3/4,4/5}
{
\filldraw[teal] (\x,\y) circle (1.5mm);
}

\end{scope}

\begin{scope}[xshift=19.5cm,yshift=0cm,local bounding box=even]

\node[above] at (5,9.5) {\large $n$ is even};
\node[below] at (5,-1) {\large $\pi$};

\draw[thick,rounded corners, blue!30,fill=blue!5] (0,0) ++(1,0)
+(9.2,9.5) -- 
+(-0.5,-0.2) -- 
+(+9.2,-0.2);

\draw[help lines] ++(1,0) +(0,0) grid +(9,9);

\draw[thin, <->] (0,9.5) |- (10.5,0);

\foreach \x in {0,1,2,3,4,5,6,7,8,9,10} 
{
\node[below] at (\x,0) {\tiny $\x$};
}

\foreach \x in {0,1,2,3,4,5,6,7,8,9} 
{
\node[left] at (0,\x) {\tiny $\x$};
}

\foreach \x/\y in {1/0,3/0,5/3,7/3,9/6}
{
\filldraw[red!90!black] (\x,\y) circle (1.5mm);
}

\foreach \x/\y in {2/1,4/3,6/4,8/5,10/8}
{
\filldraw[teal] (\x,\y) circle (1.5mm);
}

\end{scope}

\begin{scope}[xshift=18cm,yshift=-14cm,local bounding box=even-odd]

\node[below] at (2.5,-1) {\large $\pi_o$};

\draw[thick,rounded corners, blue!30,fill=blue!5] (0,0)
+(5.2,10.6) -- 
+(-0.2,-0.2) -- 
+(+5.2,-0.2);

\draw[help lines] +(0,0) grid +(5,10);

\draw[thin, <->] (0,10.5) |- (5.5,0);

\foreach \x in {1,3,5,7,9,11} 
{
\node[below] at (\x/2-1/2,0) {\tiny $\x$};
}

\foreach \x in {0,1,2,3,4,5,6,7,8,9,10} 
{
\node[left] at (0,\x) {\tiny $\x$};
}

\draw[red,thick] (0,0)
-| (1,0)
-| (2,3)
-| (3,3)
-| (4,6)
-| (5,10);

\foreach \x/\y in {0/0,1/0,2/3,3/3,4/6}
{
\filldraw[red!90!black] (\x,\y) circle (1.5mm);
}

\end{scope}

\begin{scope}[xshift=25cm,yshift=-14cm,local bounding box=even-even]

\node[below] at (3,-1) {\large $\pi_e$};

\draw[thick,rounded corners, blue!30,fill=blue!5] (0,0)
+(6.2,11.6) -- 
+(-0.2+0.5,-0.2) -- 
+(+6.2,-0.2);

\draw[help lines] ++(1,0) +(0,0) grid +(5,11);

\draw[thin, <->] (0,11.5) |- (6.5,0);

\foreach \x in {0,2,4,6,8,10,12} 
{
\node[below] at (\x/2,0) {\tiny $\x$};
}

\foreach \x in {0,1,2,3,4,5,6,7,8,9,10,11} 
{
\node[left] at (0,\x) {\tiny $\x$};
}

\draw[teal,thick] (1,0)
-| (1,1)
-| (2,3)
-| (3,4)
-| (4,5)
-| (5,8)
-| (6,11);

\foreach \x/\y in {1/1,2/3,3/4,4/5,5/8}
{
\filldraw[teal] (\x,\y) circle (1.5mm);
}

\end{scope}

\end{tikzpicture}
}
\end{center}
\caption{An example of the bijection between a subexcedent faro word and pairs of ternary paths as in the proof of Theorem~\ref{prm-dum}.}  
\label{fig:map-subexc-path}
\end{figure}


Now consider a north-east lattice path (as in Figure~\ref{fig:map-subexc-path} from $(0,0)$ to $(k,a_{k+1})$ through vertices $(0,a_1),(1,a_2),\dots,(k-1,a_k)$ in that order so that each vertex is joined to the next one by a (possibly empty) sequence of east steps followed by a (possibly empty) sequence of north steps. In other words, consider the path
\begin{equation} \label{eq:subexc-path}
N^{a_1},E,N^{a_2-a_1},E,N^{a_3-a_2},E,\dots,E,N^{a_{k+1}-a_k}
\end{equation}
from $(0,0)$ to $(k,a_{k+1})$, where $E=(1,0)$ is the unit east step and $N=(0,1)$ is the unit north step. Then this path lies on or below the line $y=2x$ for $\pi_o$ and on or below the line $y=2x+1$ for $\pi_e$, and each such path corresponds to a unique $\pi_o$ or a unique $\pi_e$.

Moreover, notice that if $n$ is even, then
\[
\begin{split}
    \left(\left\lfloor\frac{n+1}{2}\right\rfloor,\,2\left\lfloor\frac{n+1}{2}\right\rfloor\right)&=\left(\left\lfloor\frac{n}{2}\right\rfloor,\,n\right)\\
    \left(\left\lfloor\frac{n}{2}\right\rfloor,\,2\left\lfloor\frac{n}{2}\right\rfloor+1\right)&=\left(\left\lfloor\frac{n+1}{2}\right\rfloor,\,n+1\right),
\end{split}
\]
and if $n$ is odd, then
\[
\begin{split}
    \left(\left\lfloor\frac{n+1}{2}\right\rfloor,\,2\left\lfloor\frac{n+1}{2}\right\rfloor\right)&=\left(\left\lfloor\frac{n+1}{2}\right\rfloor,\,n+1\right)\\
    \left(\left\lfloor\frac{n}{2}\right\rfloor,\,2\left\lfloor\frac{n}{2}\right\rfloor+1\right)&=\left(\left\lfloor\frac{n}{2}\right\rfloor,\,n\right).
\end{split}
\]
It is easy to see now that the pair of paths thus obtained for $\pi_o$ and $\pi_e$ are in bijection with the pair of paths in the proof of the \cite[Theorem 3.5]{Bur} (see also \cite[Figure 6]{Bur}), which yields a bijection between the
 subexcedent faro words of size $n$ and $2\mbox{-}1\mbox{-}4\mbox{-}3$-avoiding Dumont permutations of the second kind of size $2n$.
\end{proof}

The enumeration of  subexcedent faro words may be refined by considering some natural statistics on such words. Together with the bijection of Theorem \ref{prm-dum} to pairs of ternary paths (or $2$-Dyck paths), a recent result \cite{Bur20} lets us find several equidistributed statistics on the odd-position and even-position subsequences of subexcedent faro words.

Recall that a \emph{ternary} (or \emph{$2$-Dyck}) path is a sequence of unit steps $u=(1,1)$ and $d=(1,-2)$ starting at $(0,0)$ and staying in the first quadrant. A \emph{peak} of a $2$-Dyck path is an $ud$-block in that path, as well as the vertex between the two steps. Likewise, a \emph{double descent} of a $2$-Dyck path is a $dd$-block in that path, as well as the vertex between the two steps. Define the following statistics on $2$-Dyck paths:
\begin{itemize}
    \item $\pk_0$, the number of peaks at even height,
    \item $\pk_1$, the number of peaks at odd height,
    \item $\dd$, the number of double descents.
\end{itemize}

Then the following results hold.
\begin{thm}[\cite{Bur20}] \label{thm:peak-heights} ~\\[-\baselineskip]
\begin{itemize}
\item On $2$-Dyck paths ending at height $0$, the tristatistic $(\pk_0 - \m1,\pk_1,\dd)$ is jointly equidistributed with any of its permutations.

\item On $2$-Dyck paths ending at height $1$, the bistatistics $(\pk_0,\pk_1)$ and $(\pk_1,\pk_0)$ are jointly equidistributed.
\end{itemize}
\end{thm}

For a subexcedent faro word $\pi$ of length $2n$, define the following statistics on its odd-position and even-position subsequences $\pi_o$ and $\pi_e$:
\begin{itemize}

\item $\eOdis(\pi)$, the number of distinct \emph{positive} even letters in $\pi_o$ (we exclude $0$ since $\pi_o$ and $\pi$ always start with $0$);

\item $\oOdis(\pi)$, the number of distinct odd letters in $\pi_o$;

\item $\aOrpt(\pi)=\{i\in[n-1]\,\mid\,\pi(2i-1)=\pi(2i+1)\}$, the number of letter repetitions in $\pi_o$ (the ``a'' in $\aOrpt$ stands for ``any parity'');

\item $\eEdis(\pi)$, the number of distinct even letters in $\pi_e$;

\item $\oEdis(\pi)$, the number of distinct odd letters in $\pi_e$.

\end{itemize}

Then we have the following result.
\begin{thm} \label{thm:dist-rpt-stats}
On subexcedant faro words of length $n$,
\begin{itemize}
\item the tristatistic $(\eOdis,\oOdis,\aOrpt)$ is jointly equidistributed with any of its permutations.

\item the bistatistics $(\eEdis,\oEdis)$ and $(\oEdis,\eEdis)$ are jointly equidistributed.
\end{itemize}
\end{thm}

\begin{proof}
For each of $\pi_o$ and $\pi_e$, define $k$ and $a_1,a_2,\dots,a_k,a_k+1$ as in the proof of Theorem \ref{prm-dum}, and let
\[
P=N^{a_1},E,N^{a_2-a_1},E,N^{a_3-a_2},E,\dots,E,N^{a_{k+1}-a_k}
\]
be the corresponding north-east path as in \eqref{eq:subexc-path} (when needed, we will distinguish the paths obtained from $\pi_o$ and $\pi_e$ as $P_o$ and $P_e$, respectively). Map $P$ to a lattice path obtained by reversing $P$ and mapping unit steps $N\mapsto u=(1,1)$ and $E\mapsto d=(1,-2)$. In other words, consider the map
\[
\phi: P\mapsto\phi(P)=u^{a_{k+1}-a_k},d,u^{a_k-a_{k-1}},d,\dots,d,u^{a_2-a_1},d,u^{a_1},
\]
where $\phi(P)$ starts at $(0,0)$. Recall that $P$ starts at $(0,0)$, stays in the first quadrant on or below $y=2x$ for $\pi_o$ and $y=2x+1$ for $\pi_e$, and ends on $y=2x$ for $\pi_o$ and $y=2x+1$ for $\pi_e$. Therefore, it is easy to see that $\phi(P)$ stays in the first quadrant and ends at height $0$ for $\pi_o$ and at height $1$ for $\pi_e$. Moreover, each distinct letter of $\pi_o$ or $\pi_e$ (except for $0$ in $\pi_o$) corresponds to a block $EN$ in the corresponding path $P$, which in turn corresponds to a block $ud$ of $\phi(P)$, i.e. to a peak of $\phi(P)$.

Furthermore, a repetition of a letter in positions $i$ and $i+1$ of $\pi_o$ means that $a_{i+1}=a_{i}$, and thus the $i$-th and $(i+1)$-st steps $E$ in $P$ are adjacent, which in turn corresponds to a block $dd$ in $\phi(P)$. Therefore, $\aOrpt(\pi)=\dd(P_o)$.

Let $\ell$ be one of distinct letters of in $\pi_o$ or $\pi_e$ (for $\pi_o$, also assume $\ell>0$). Suppose its rightmost occurrence is in position $j$. Then there are $k+1-\ell$ east steps and $a_{k+1}-a_{\ell}$ north steps in path $P$ to the right of that point, so the height of the corresponding peak in $\phi(P)$ is 
\[
a_{k+1}-a_{\ell}-2(k+1-\ell) \equiv a_{k+1}-a_{\ell}\!\!\!\pmod 2 \equiv a_\ell\!\!\!\pmod 2 + a_{k+1}\!\!\!\pmod 2.
\]
It follows that, on $\pi_o$ $(\eOdis,\oOdis)(\pi)=(\pk_0-1,\pk_1)(\phi(P_o))$ if $a_{k+1}$ is even, and $(\eOdis,\oOdis)(\pi)=(\pk_1,\pk_0-1)(\phi(P_o))$ if $a_{k+1}$ is odd. Likewise, $(\eEdis,\oEdis)(\pi)=(\pk_0,\pk_1)(\phi(P_e))$ if $a_{k+1}$ is even, and $(\eEdis,\oEdis)(\pi)=(\pk_1,\pk_0)(\phi(P_e))$ if $a_{k+1}$ is odd. However, the two statistics on the right-hand side of the equations are jointly equidistributed in each case by Theorem \ref{thm:peak-heights}, and thus the parity of $a_{k+1}$ is immaterial in each case.
\end{proof}

From Corollary 1.12 and Equation (2.7) of \cite{Bur20}, we can also determine the joint distribution of all the statistics we defined on subexcedent faro words. For this result, we let $n_o=\left\lfloor\frac{n+1}{2}\right\rfloor$ and $n_e=\left\lfloor\frac{n}{2}\right\rfloor$ (so $n_o+n_e=n$). We also let $\aErpt(\pi)$ be the number of letter repetitions in $\pi_e$, i.e. $\aErpt(\pi)=\{i\in[n-1]\,\mid\,\pi(2i)=\pi(2i+2)\}$.

\begin{cor}
The number of subexcedent faro words $\pi$ of length $n$ such that
\[
(\eOdis,\oOdis,\aOrpt,\eEdis,\oEdis,\aErpt)(\pi)=(r_1,r_2,r_2,r_4,r_5,r_6)
\]
is
\[
\frac{1}{n_o}\binom{n_o}{r_1}\binom{n_o}{r_2}\binom{n_o}{r_3}\frac{r_4+r_5}{n_e(n_e+1)}\binom{n_e+1}{r_4}\binom{n_e+1}{r_5}\binom{n_e}{r_6}.
\]
\end{cor}

Note also that $r_1+r_2+r_3=n_o-1$ and $r_4+r_5+r_6=n_e$.

\section{Acknowledgments} We would like to greatly thank anonymous referees for their helpful comments and suggestions.



\bibliographystyle{plain}
\bibliography{bib}

\begin{thebibliography}{10}

\bibitem{Band2}
A.~Asinowski, A.~Bacher, C.~Banderier, and B.~Gittenberger.
\newblock Analytic combinatorics of lattice paths with forbidden patterns, the
  vectorial kernel method, and generating functions for pushdown automata.
\newblock {\em Algorithmica}, 82(3):386--428, 2020.

\bibitem{Band1}
A.~Asinowski, C.~Banderier, and V.~Roitner.
\newblock Generating functions for lattice paths with several forbidden
  patterns.
\newblock {\em Séminaire Lotharingien de Combinatoire --- Proceedings of the
  32nd International Conference on "Formal Power Series and Algebraic
  Combinatorics"}, 84B, 2020.

\bibitem{Bach}
A.~Bacher, A.~Bernini, L.~Ferrari, B.~Gunby, R.~Pinzani, and J.~West.
\newblock The {D}yck pattern poset.
\newblock {\em Discrete Mathematics}, 321:12--23, 2014.

\bibitem{patdd}
J.-L. Baril, R.~Genestier, and S.~Kirgizov.
\newblock Pattern distributions in {D}yck paths with a first return
  decomposition constrained by height.
\newblock {\em Discrete Mathematics}, 342(9):111995, 2020.

\bibitem{des}
J.-L. Baril and V.~Vajnovszki.
\newblock Popularity of patterns over $d$-equivalence classes of words and
  permutations.
\newblock {\em Theoretical Computer Science}, 814:249--258, April 2020.

\bibitem{bar}
M.~Barnabei, F.~Bonetti, and M.~Silimbani.
\newblock Restricted involutions and {M}otzkin paths.
\newblock {\em Advances in Applied Mathematics}, 47(1):102--115, 2011.

\bibitem{barnes}
C.~S. Barnes.
\newblock Enumeration of the distinct shuffles of permutations.
\newblock {\em {DMTCS Proceedings vol. AK, 21st International Conference on
  Formal Power Series and Algebraic Combinatorics (FPSAC 2009)}}, pages
  155--166, 2009.

\bibitem{Ber}
J.~Bertrand.
\newblock Solution d'un problème.
\newblock {\em Comptes Rendus de l'Académie des Sciences}, 105:369, 1887.

\bibitem{biane}
P.~Biane.
\newblock Permutations suivant le type d'excédance et le nombre d'inversions
  et interprétation combinatoire d'une fraction continue de {H}eine.
\newblock {\em European Journal of Combinatorics}, 14(4):277--284, July 1993.

\bibitem{Bon}
M.~B{\'{o}}na.
\newblock Surprising symmetries in objects counted by {C}atalan numbers.
\newblock {\em The Electronic Journal of Combinatorics}, 19(1), March 2012.

\bibitem{Bur20}
A.~Burstein.
\newblock Distribution of peak heights modulo $k$ and double descents on
  $k$-{D}yck paths.
\newblock (submitted), \url{https://arxiv.org/pdf/2009.00760.pdf}, 2020.

\bibitem{Bur}
A.~Burstein, S.~Elizalde, and T.~Mansour.
\newblock Restricted {D}umont permutations, {D}yck paths, and noncrossing
  partitions.
\newblock {\em Discrete Mathematics}, 306(22):2851--2869, 2006.

\bibitem{bona}
M.~Bóna.
\newblock {\em A walk through combinatorics: an introduction to enumeration and
  graph theory}.
\newblock World Scientific Pub, Hackensack, NJ, 2nd ed edition, 2006.

\bibitem{deu}
E.~Deutsch.
\newblock Dyck path enumeration.
\newblock {\em Discrete Mathematics}, 204(1-3):167--202, 1999.

\bibitem{Dia}
P.~Diaconis, R.L. Graham, and W.~M. Kantor.
\newblock The mathematics of perfect shuffles.
\newblock {\em Advances in Applied Mathematics}, 4(2):175--196, June 1983.

\bibitem{fra}
J.~Fran{\c{c}}on and G.~Viennot.
\newblock Permutations selon leurs pics, creux, doubles mont{\'e}es et double
  descentes, nombres d'{E}uler et nombres de {G}enocchi.
\newblock {\em Discrete Mathematics}, 28(1):21--35, 1979.

\bibitem{hac}
B.~Hackl, C.~Heuberger, and H.~Prodinger.
\newblock {Counting Ascents in Generalized {D}yck Paths}.
\newblock In {\em 29th International Conference on Probabilistic, Combinatorial
  and Asymptotic Methods for the Analysis of Algorithms (AofA 2018)}, volume
  110, pages 26:1--26:15, Dagstuhl, Germany, 2018.

\bibitem{Homb}
C.~Homberger.
\newblock Expected patterns in permutation classes.
\newblock {\em The Electronic Journal of Combinatorics}, 19(3):$\#$P43, 2012.

\bibitem{Kit}
S.~Kitaev.
\newblock {\em Patterns in permutations and words}.
\newblock Springer Science \& Business Media, 2011.

\bibitem{Kra01}
C.~Krattenthaler.
\newblock Permutations with restricted patterns and {D}yck paths.
\newblock {\em Advances in Applied Mathematics}, 27:510--530, 2001.

\bibitem{lot}
M.~Lothaire.
\newblock {\em Combinatorics on words}, volume~17 of {\em Encyclopedia of
  Mathematics and Its Applications}.
\newblock Addison-Wesley, 1983.

\bibitem{mansour}
T.~Mansour.
\newblock Statistics on {D}yck paths.
\newblock {\em J. Integer Seq}, 9(1):06–1, 2006.

\bibitem{Mor}
S.B. Morris.
\newblock The basic mathematics of the faro shuffle.
\newblock {\em Pi Mu Epsilon Journal}, 6(2):85--92, 1975.

\bibitem{gfun}
B.~Salvy and P.~Zimmermann.
\newblock Gfun: A {M}aple package for the manipulation of generating and
  holonomic functions in one variable.
\newblock {\em ACM Trans. Math. Softw.}, 20(2):163–177, June 1994.

\bibitem{sim}
R.~Simion and F.~Schmidt.
\newblock Restricted permutations.
\newblock {\em European J. Combin}, 6:383--406, 1985.

\bibitem{Slo}
N.J.A. Sloane et~al.
\newblock {The On-line Encyclopedia of Integer Sequences}.
\newblock \url{https://oeis.org}, 2020.

\bibitem{sta}
R.~Stanley.
\newblock {\em Enumerative Combinatorics, Vol. 1.}
\newblock Cambridge University Press, 1997.

\bibitem{Whi}
W.A. Whitworth.
\newblock Arrangements of $m$ things of one sort and $n$ things of another
  sort, under certain conditions of priority.
\newblock {\em Messenger of Math}, 8:105--114, 1878.

\end{thebibliography}
\end{document}